\documentclass{elsarticle}

\biboptions{sort&compress,numbers}
\usepackage[nodots]{numcompress}
\usepackage{lineno,url}
\usepackage{amsmath,amsfonts,amssymb}
\usepackage{amssymb,amsmath,amsthm}
\usepackage{overpic}
\usepackage{contour}\contourlength{0.7pt}
\usepackage{tikz}\usetikzlibrary{shapes,arrows}
\usepackage{versions}
\usepackage{natbib}
\usepackage{algorithm, algorithmic}

\newtheoremstyle{mytheorem}{5pt plus 5pt minus 3pt}{4pt plus 3pt minus 1.5pt}
	{\itshape}{}{\bfseries}{.}{1ex plus 1ex minus .5ex}{}
\newtheoremstyle{mydef}{5pt plus 5pt minus 3pt}{4pt plus 3pt minus 1.5pt}
	{}{0pt}{\bfseries}{.}{1ex plus 1ex minus .5ex}{}
\newtheoremstyle{myremark}{5pt plus 5pt minus 3pt}{4pt plus 3pt minus 1.5pt}
	{}{0pt}{\itshape}{.}{1ex plus 1ex minus .5ex}{}
\theoremstyle{mytheorem}

\newtheorem{lemma}{Lemma}[section]
\newtheorem{theorem}{Theorem}[section]
\newtheorem{cor}{Corollary}
\newtheorem{rem}{Remark}
\theoremstyle{mydef}

\theoremstyle{myremark}

\newcommand{\bc}{\mathbf c}
\newcommand{\bq}{\mathbf q}

\newcommand{\D}{\mathbf D}
\newcommand{\Q}{\mathcal Q}
\newcommand{\tf}{\scriptsize}

\def\Acaption#1#2{\caption{#2}\vspace*{-#1}}

\definecolor{Mgreen}{RGB}{34,139,34}
\definecolor{blau}{rgb}{0.15,0.2,0.5}
\definecolor{gray}{rgb}{0.5,0.5,0.5}
\definecolor{drot}{rgb}{0.7,0,0.1}
\definecolor{gelb}{rgb}{.55,.40,.1}
\definecolor{magenta}{rgb}{1.,0.,1.}
\definecolor{cyan}{rgb}{0.,1.,1.}
\definecolor{green}{rgb}{0.,1.,0.}
\definecolor{Morange}{rgb}{1.,0.5,0.}

\begin{document}

\begin{frontmatter}
\title{Gaussian quadrature rules for $C^1$ quintic splines}

\author[NumPor]{Michael Barto\v{n}\corref{cor1}}
\ead{Michael.Barton@kaust.edu.sa}
\author[VCC]{Rachid Ait-Haddou}
\ead{rachid.aithaddou@kaust.edu.sa}
\author[NumPor]{Victor Manuel Calo}
\ead{Victor.Calo@kaust.edu.sa}

\cortext[cor1]{Corresponding author}

\address[NumPor]{Numerical Porous Media Center, King Abdullah University of Science and Technology, Thuwal 23955-6900, Kingdom of Saudi Arabia}
\address[VCC]{Visual Computing Center, King Abdullah University of Science and Technology, \\ Thuwal 23955-6900, Kingdom of Saudi Arabia}

\begin{abstract}
We provide explicit expressions for quadrature rules
on the space of $C^1$ quintic splines with uniform knot sequences over finite domains.
The quadrature nodes and weights are derived via an explicit recursion
that avoids an intervention of any numerical solver and the rule is optimal, that is,
it requires the minimal number of nodes. For each of $n$ subintervals, generically, only two nodes
are required which reduces the evaluation cost by $2/3$ when compared to the classical Gaussian
quadrature for polynomials. Numerical experiments show fast convergence,
as $n$ grows, to the ``two-third'' quadrature rule of Hughes et al. \cite{Hughes-2010} for infinite domains.
\end{abstract}

\begin{keyword}
Gaussian quadrature, quintic splines, Peano kernel, B-splines
\end{keyword}

\end{frontmatter}


\section{Introduction}\label{intro}

Numerical quadrature has been of interest for decades due to its
wide applicability in many fields spanning collocation methods \cite{Sloan-1988},
integral equations \cite{Atkinson-1976}, finite elements methods \cite{Solin-2003}
and most recently, isogeometric analysis \cite{Cottrell-2009}. It is also a preferable
tool for high-speed solution frameworks \cite{petiga,PetigaPaper-2013} as it is computationally very cheap and robust
when compared to classical integration methods \cite{Gautschi-1997}.

A quadrature rule, or shortly a quadrature, is said to be an \textit{$m$-point rule},
if $m$ evaluations of a function $f$ are needed to approximate its weighted integral
over an interval $[a,b]$
\begin{equation}\label{eq:GaussQuad}
\int_a^b \omega(x) f(x) \, \mathrm{d}x = \sum_{i=1}^{m} \omega_i f(\tau_i) + R_{m}(f),
\end{equation}
where $\omega$ is a fixed non-negative \emph{weight function} defined over $[a,b]$.
Typically, the rule is required to be \emph{exact}, that is, $R_m(f) \equiv 0$
for each element of a predefined linear function space $\mathcal{L}$.
Moreover, the rule is said to be \emph{optimal} if $m$ is the \emph{minimal} number of
\textit{nodes} $\tau_i$ at which $f$ has to be evaluated.

In the literature, the term \emph{optimality} may also refer to the approximation error that
the quadrature rule produces. That is, the number of nodes is given and
their layout is sought such that it minimizes the error for a given class of
functions. K\"{o}hler and Nikolov \cite{NikolovKoehler-1995,NikolovKoehler2-1995} showed that
the Gauss-type quadrature formulae associated with spaces of spline functions
with equidistant knots are asymptotically optimal in non-periodic Sobolev
classes. This is a motivation for studying Gauss-type quadrature formulae for
spaces of spline functions, in particular, with equidistant knots.
In this paper, by \emph{optimal} we exclusively mean quadrature rules with the minimal number of nodes.

In the case when $\mathcal{L}$ is the linear space of polynomials of degree at most $2m-1$,
then the $m$-point Gaussian quadrature rule \cite{Gautschi-1997} is both exact and optimal.
The Gaussian nodes are the roots of the orthogonal polynomial $\pi_{m}$ where
$(\pi_0,\pi_1,\ldots,\pi_m,\ldots )$ is the sequence of orthogonal polynomials with
respect to the measure $\mu(x) = \omega(x)dx$.
Typically, the nodes of the Gaussian quadrature rule
have to be computed numerically.
which becomes expensive, especially for high-degree polynomials.

Naturally, getting more degrees of freedom is not achieved by using a higher polynomial
degree, but rather by using polynomial pieces, i.e., \emph{splines}.
Additionally to the polynomial case, the interval of interest is provided
by a non-decreasing sequence of points known as a \emph{knot sequence} (or vector), points
where the resulting spline is considered to have a lower continuity.
The knot sequence defines the \emph{local support} of each basis function, that is, each one acts only
locally on a particular subinterval of the domain, and--depending on the knots multiplicities--
spans a particular number of knots.
We refer to \cite{deBoor-1972,Elber-2001,Hoschek-2002-CAGD} for a detailed introduction to splines.

Regarding the quadrature rules for splines, Micchelli and Pinkus \cite{Micchelli-1977}
derived the optimal number of quadrature nodes and specified the range of intervals,
the knot sequence subintervals, that contain at least one node.
There are two main difficulties compared to the polynomial case:
firstly, the optimal quadrature rule is not in general guaranteed to be unique, e.g.,
when the boundary constraints are involved,
and, secondly, \cite{Micchelli-1977} determines
only a \emph{range} of intervals, i.e., each node has several potential subintervals to lie in.
The latter issue is crucial as one cannot apply even expensive numerical solvers,
because the \emph{algebraic} system to solve is not known. For each assumed layout of nodes, one
would have to solve a particular algebraic system using e.g., \cite{SolverGershon-2001,Patrikalasis-1993,SolverET-2012}.
However, the number of eventual systems
grows exponentially in the number of subintervals and therefore such an approach is not feasible.
Instead, theorems that derive exact layouts of the nodes are essential. Our work in this paper
contributes such a theoretical result for a particular space of splines.

The quadrature rules for splines differ depending on the particular space of interest $S_{d,c}$, where $d$ is the \emph{degree} and $c$ refers to
\emph{continuity}. For cases with lower continuity, the ``interaction'' between polynomial pieces is lower and hence, naturally,
a higher number of nodes is required for the optimal quadrature rule.
The choice of the domain can bring a significant simplification.
Whilst there are few rules that are exact and optimal over a finite domain,
their counterparts are known when the integration domain is the whole \emph{real line} \cite{Hughes-2010}.
The \emph{half-point rules} of Hughes et al. are independent of the polynomial degree and the ``half'' indicates
that the number of quadrature points is roughly half the number of basis functions.

The half-point rules can be altered even for spaces with lower continuity, e.g., for $S_{4,1}$
an optimal rule was also derived in \cite{Hughes-2010}.
The rule is called a ``two-third rule'' as it requires only two evaluations per subinterval
whilst the classical Gaussian rule for polynomials needs three nodes.
However, these rules are exact only over the real line.
Because in most applications a \emph{finite domain} is needed,
additional nodes have to be added to satisfy the boundary constraints and a numerical solver
has to be employed \cite{Hughes-2012}. We emphasize that we focus here only on optimal rules
as there are many schemes that introduce redundant nodes in order to overcome the problem with a finite interval,
see \cite{Hughes-2012} and the references cited therein.

Regarding optimal rules over finite domains,
Nikolov \cite{Nikolov-1996} proved the unique layout of nodes of the quadrature rule for $S_{3,1}$ with
uniform knot sequences and derived a recursive algorithm that computes the nodes and weights in
a closed form. Recently \cite{Quadrature31-2014}, the result was generalized for $S_{3,1}$ over
a special class of \emph{non-uniform} knot sequences, called symmetrically-stretched.
The rules possess the three desired properties, i.e., they are exact, optimal and act in a closed form fashion,
without intervention of any numerical solver.

We emphasize that the computation of the nodes and weights of the optimal spline
quadrature, also called \emph{Gaussian}, is rather a challenging problem as
one has to first derive the \emph{correct layout} of the nodes and then, typically, to solve
non-linear systems of algebraic equations. For higher degrees, the use of a numerical solver
seems to be unavoidable.

In this work, we derive a quadrature rule
for spaces of $C^1$ quintic splines, $S_{5,1}$,
with uniform knot sequences over finite domains.
The rule is exact, optimal, and--even though the degree is five--explicit.
We also show numerically, when the number of subintervals grows,
that the rule rapidly converges to the ``two-third'' rule of Hughes \cite{Hughes-2010}.

The rest of the paper is organized as follows.
In Section~\ref{sec:quad}, we recall some basic properties of $S_{5,1}$
and derive their Gaussian quadrature rules. In Section~\ref{sec:err}, the error estimates are given
and Section~\ref{sec:ex} shows some numerical experiments that validate the theory proposed in this work. Finally,
possible extensions of our method are discussed in Section~\ref{sec:con}.

\section{Gaussian quadrature formulae for $C^1$ quintic splines}\label{sec:quad}


In this section we state a few basic properties of $S_{5,1}$ splines and derive
explicit formulae for computing quadrature nodes and weights for spline spaces with uniform knot sequences over a finite domain.
Throughout the paper, $\pi_{d}$ denotes the linear space of polynomials
of degree at most $d$ and $[a,b]$ is a non-trivial real \textit{compact} interval.

\subsection{$C^1$ quintic splines with uniform knot sequences}\label{ssec:spline}

We detail several properties of spline basis functions. We
consider a uniform partition $\mathcal{X}_n = (a=x_0,x_1,...,x_{n-1},x_{n} = b)$ of the interval $[a,b]$
with $n$ subintervals and define $h := \frac{1}{n} = x_{k} -  x_{k-1}$ for all $k=1, \dots, n$.
We denote by $S^{n}_{5,1}$ the linear space of $C^1$ quintic splines over a uniform
knot sequence $\mathcal{X}_n =(a=x_0,x_1,...,x_n=b)$
\begin{equation}\label{eq:family}
S^{n}_{5,1} = \{ f\in C^{1}[a,b]: f|_{(x_{k-1},x_{k})} \in \pi_5, k=1,...,n\}.
\end{equation}
The dimension of the space $S^{n}_{5,1}$ is $4n+2$.
\begin{rem}
In the B-spline literature \cite{deBoor-1972,Hoschek-2002-CAGD,Elber-2001}, the knot sequence is
usually written with knots' multiplicities.
However, in the isogeometric analysis literature, see e.g., \cite{Adel-2014,Buffa-2010},
the knot vector is usually split into a vector carrying the partition of the interval
and a vector containing continuity information (knot multiplicity).
As in this paper the multiplicity is
always four at every knot, we follow the latter notation and, throughout the paper,
write $\mathcal{X}_n$ without multiplicity, i.e., $x_k<x_{k+1}$, $k=0,\dots, n-1$.
\end{rem}

 \begin{figure*}
 \hfill
 \begin{overpic}[width=.69\columnwidth]{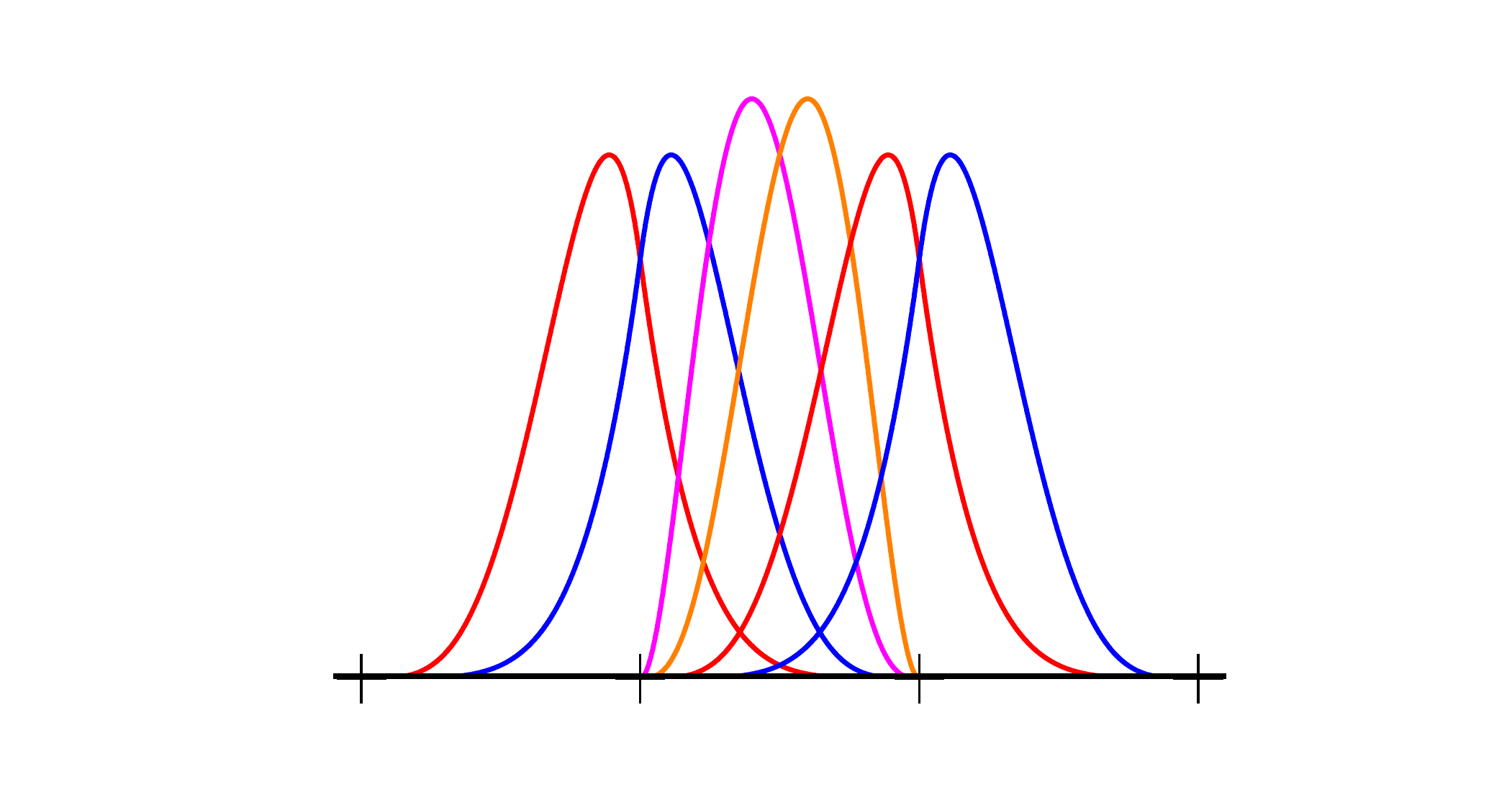}
    \put(20,3){$x_{k-2}$}
    \put(36,3){$x_{k-1}$}
    \put(60,3){$x_{k}$}
    \put(75,3){$x_{k+1}$}
    \put(20,16){\color{red}{$D_{4k-3}$}}
    \put(70,20){\color{blue}{$D_{4k+2}$}}
    \put(36,45){\color{magenta}{$D_{4k-1}$}}
    \put(55,45){\color{Morange}{$D_{4k}$}}
  \end{overpic}
 \hfill
 \vspace{-4pt}
 \Acaption{1ex}{Four consecutive knots $x_{k-2}, \dots, x_{k+1}$ of a uniform knot sequence, each of multiplicity four. Six non-normalized spline
 basis functions $D_{4k-3}, \dots, D_{4k+2}$ with non-zero support on $[x_{k-1},x_{k}]$ are displayed.}\label{fig:C1Splines}
 \end{figure*}

Similarly to \cite{Nikolov-1996,Quadrature31-2014}, we find it convenient to work with the
non-normalized B-spline basis. To define the basis, we extend our knot sequence
$\mathcal{X}_n$ with two extra knots outside the interval $[a,b]$ in a uniform fashion
\begin{equation}\label{boundary_condition}
x_{-1} = x_{0}-h
\quad \textnormal{and} \quad
x_{n+1} =x_{n}+h.
\end{equation}
The choice of $x_{-1}$ and $x_{n+1}$ allows us to
simplify expressions in Section~\ref{ssec:quad}, but this setting does not affect the quadrature
rule derived later in Theorem~\ref{theo:Quadrature}.
We follow \cite{Schoenberg-1966} and denote by $\D = \{D_i\}_{i=1}^{4n+2}$ the basis of $S^{n}_{5,1}$ where
\begin{equation}\label{basis1}
\begin{split}
D_{4k-3}(t) & = [x_{k-2},x_{k-2},x_{k-1},x_{k-1},x_{k-1},x_{k-1},x_{k}](.-t)_{+}^5, \\
D_{4k-2}(t) & = [x_{k-2},x_{k-1},x_{k-1},x_{k-1},x_{k-1},x_{k},x_{k}](.-t)_{+}^5, \\
D_{4k-1}(t) & = [x_{k-1},x_{k-1},x_{k-1},x_{k-1},x_{k},x_{k},x_{k}](.-t)_{+}^5, \\
D_{4k}(t) & = [x_{k-1},x_{k-1},x_{k-1},x_{k},x_{k},x_{k},x_{k}](.-t)_{+}^5.
\end{split}
\end{equation}
where $[.]f$ stands for the divided difference and $u_{+} = \max(u,0)$ is the truncated
power function.
The direct computation of the divided differences gives the following explicit expressions
for $t \in [x_{k-2},x_{k-1}]$
\begin{equation}\label{basis2a}
\renewcommand{\arraystretch}{1.4}
\begin{array}{rcl}
D_{4k-3}(t) & = & \frac{(t - x_{k-2})^4(x_{k} + 8x_{k-1} - 9t)}{4h^6}, \\
D_{4k-2}(t) & = & \frac{(t - x_{k-2})^5}{4h^6},
\end{array}
\end{equation}
and for $t \in [x_{k-1},x_{k}]$
\begin{equation}\label{basis2b}
\renewcommand{\arraystretch}{1.4}
\begin{array}{rcl}
D_{4k-3}(t) & = & \frac{(x_{k}-t)^5}{4h^6}, \\
D_{4k-2}(t) & = & \frac{(x_{k}-t)^4(x_{k-2} + 8x_{k-1} - 9t)}{4h^6}, \\
D_{4k-1}(t) & = & \frac{10(t - x_{k-1})^2 (x_{k}-t)^3}{h^6},\\
D_{4k}(t)   & = & \frac{10(t - x_{k-1})^3 (x_{k}-t)^2}{h^6}.
\end{array}
\end{equation}
The functions have the following pattern: six basis functions
$D_{4k-3}, \dots,$ $D_{4k+2}$ have non-zero support on $[x_{k-1},x_k]$, moreover, two
of them, $D_{4k-1}$ and $D_{4k}$, act only $[x_{k-1},x_k]$ and are scaled Bernstein basis functions,
see (\ref{basis2b}) and Fig.~\ref{fig:C1Splines}.

Among the basic properties of the
basis $\D$, we need to recall the fact that
$D_{4k+2}(t) \leq D_{4k+1}(t)$ for $t\in[x_{k-1},x_k]$ for $k=1,\ldots,n$
and are equal at the knot $x_k$, that is,
$D_{4k+1}(x_{k}) = D_{4k+2}(x_{k}) = \frac{1}{4h}$. Moreover, we have that
$D_{4k-1}(\frac{x_{k-1}+x_k}{2}) = D_{4k}(\frac{x_{k-1}+x_k}{2}) = \frac{5}{16h}$.
From (\ref{basis2a}) and (\ref{basis2b}), the integrals of the basis functions are computed
\begin{equation}\label{interiorIntegral}
I[D_k] = \frac{1}{6}\; \textnormal{for} \quad k = 3,4,\ldots,4n,
\end{equation}
where $I[f]$ stands for the integral of $f$ over the interval $[a,b]$.
The first and the last two integrals are
\begin{equation}\label{boundaryIntegral}
I[D_1] =  I[D_{4n+2}] = \frac{1}{24}
\quad \textnormal{and} \quad
I[D_2] =  I[D_{4n+1}]= \frac{1}{8}.
\end{equation}

 \begin{figure*}[!tb]
 \hfill
 \begin{overpic}[width=.69\columnwidth]{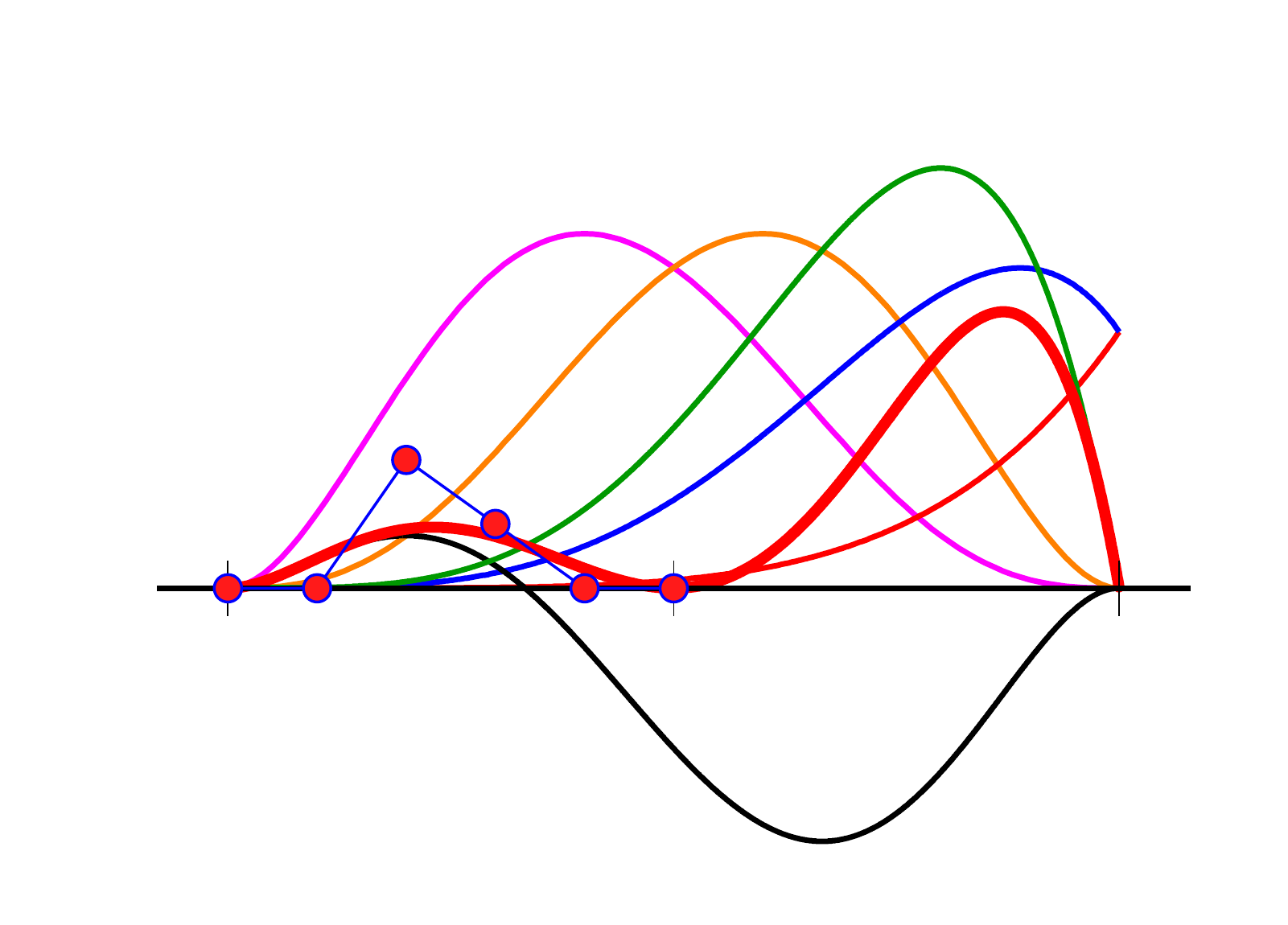}
    \put(15,20){$x_{k-1}$}
    \put(51,20){$\frac{x_{k-1}+x_{k}}{2}$}
    \put(88,20){$x_{k}$}
    \put(0,37){$[x_{k-1}+\frac{2}{10}h,\frac{1}{8h}]$}
    \put(57,33){\color{red}{$P_k$}}
    \put(60,60){\color{Mgreen}{$2D_{4k+1}-2D_{4k+2}$}}
    \put(82,52){\color{blue}{$D_{4k+2}$}}
    \put(25,50){\color{magenta}{$D_{4k-1}$}}
    \put(55,55){\color{Morange}{$D_{4k}$}}
    \put(72,7){$\frac{1}{2}D_{4k-1}-D_{4k}$}
  \end{overpic}
 \hfill
 \vspace{-4pt}
 \Acaption{1ex}{The linear blend of basis functions (\ref{blend}), $P_k$, is expressed as a B\'{e}zier curve on $[x_{k-1},\frac{x_{k-1}+x_{k}}{2}]$
 with a control point sequence (red dots) with non-negative $y$-coordinates (\ref{eq:P1}). Consequently, $P_k$ is non-negative
 on $[x_{k-1},x_{k}]$ and has a single root (of multiplicity two) at $\frac{x_{k-1}+x_{k}}{2}$.}\label{fig:Lemma1}
 \end{figure*}

A less obvious property that binds together four consecutive basis functions,
which is used later for our quadrature rule in Section~\ref{ssec:quad},
is formalized as follows
\begin{lemma}\label{lemmaBlend}
Let $\mathcal{X}_{n} = (a=x_0,x_1,...,x_n=b)$, be a uniform knot sequence and
for any $k = 1,\dots,n$ define
\begin{equation}\label{blend}
P_k(t) = 2 D_{4k+1}(t) - 2 D_{4k+2}(t) + \frac{1}{2} D_{4k-1}(t) - D_{4k}(t).
\end{equation}
Then $P_k(t)\geq 0$ for any $t \in  (x_{k-1},x_{k})$ and $P_k(t) = 0$ if and only if $t=\frac{x_{k-1}+x_{k}}{2}$.
\end{lemma}
\begin{proof}
Over an interval $(x_{k-1},x_{k})$,  the function $P_k$ is a single
quintic polynomial. Therefore, it can be expressed in terms of Bernstein basis and can be
viewed as a B\'ezier curve on a particular domain, see Fig.~\ref{fig:Lemma1}.
Looking at its shape, one cannot conclude non-negativity from the control polygon,
when considered on the whole interval $(x_{k-1},x_{k})$. Hence we define
\begin{equation}\label{eq:Qs}
\renewcommand{\arraystretch}{1.4}
\begin{array}{rcl}
P_{k}^1(t) &  = P_k(t) \quad \textnormal{on}  & [x_{k-1}, \frac{x_{k-1} + x_{k}}{2}], \\
P_{k}^2(t) &  = P_k(t) \quad \textnormal{on}  & [\frac{x_{k-1} + x_{k}}{2}, x_{k}],
\end{array}
\end{equation}
and using $h = x_{k} - x_{k-1}$ we further write
\begin{equation}\nonumber
\begin{array}{c}
P_{k}^1(t) = \sum\limits_{i=0}^{5} q_i^1 B_i^5(t), \; \textnormal{where} \;
B_i^5(t) = \binom{5}{i} \left(\frac{2t-2x_{k-1}}{h}\right)^i
\left(\frac{2x_{k-1} + h -2t}{h}\right)^{5-i}
\end{array}
\end{equation}
and analogously for $P_{k}^2$.
The conversion from monomial to Bernstein basis gives the control points $(p_0^1,\dots,p_5^1)$ of $P_k^1$
over the interval $[x_{k-1},x_{k-1}+h/2]$ as
\begin{equation}\label{eq:P1}
(p_0^1,p_1^1,p_2^1,p_3^1,p_4^1,p_5^1) = \left(0,0, \frac{1}{8h}, \frac{1}{16h}, 0, 0\right)
\end{equation}
and similarly for $P_{k}^2$ we obtain
\begin{equation}\label{eq:P2}
(p_0^2,p_1^2,p_2^2,p_3^2,p_4^2,p_5^2) = \left(0,0, \frac{1}{16h}, \frac{1}{4h}, \frac{1}{2h}, 0\right).
\end{equation}
Therefore, $P_{k}^1$ and $P_{k}^2$ are non-negative on open intervals $(x_{k-1},x_{k-1}+h/2)$ and
$(x_{k-1}+h/2, x_{k})$, respectively. Due to the fact that $(p_4^1, p_5^1) = (p_0^2, p_1^2) = (0,0)$,
the only root (with multiplicity two) of $P_k$ on $(x_{k-1},x_{k})$ is $x_{k-1}+h/2$.
\end{proof}

\begin{rem}
$D_{4k+1} - D_{4k+2}$, $D_{4k-1}$, $D_{4k}$ are all positive polynomials on $(x_{k-1},x_{k})$
and therefore there exists infinitely many non-negative blends. However,
the existence of a non-negative blend when the coefficients have to satisfy a certain constraint is not obvious
and the full impact of this particular blend with coefficients $2, \frac{1}{2}, -1$
will be seen later in Lemma~\ref{lemmag1}.
%
\end{rem}

\subsection{Gaussian quadrature formulae}\label{ssec:quad}
In this section, we derive a quadrature rule for the family $S^{n}_{5,1}$, see (\ref{eq:family}).
Similarly to \cite{Quadrature31-2014}, we derive a quadrature rule that is optimal, exact and explicit.

With respect to exactness and optimality, according to \cite{Micchelli-1977,Micchelli-1972} there exists
a quadrature rule
\begin{equation}\label{quadrature}
\Q_a^b(f) = \int_{a}^{b} f(t) \mathrm{d}t = \sum_{i=1}^{2n+1} \omega_i f(\tau_i)
\end{equation}
that is exact for every function $f$ from the space $S^{n}_{5,1}$.
The explicitness follows from the construction.
 \begin{figure*}
 \hfill
 \begin{overpic}[width=.99\columnwidth]{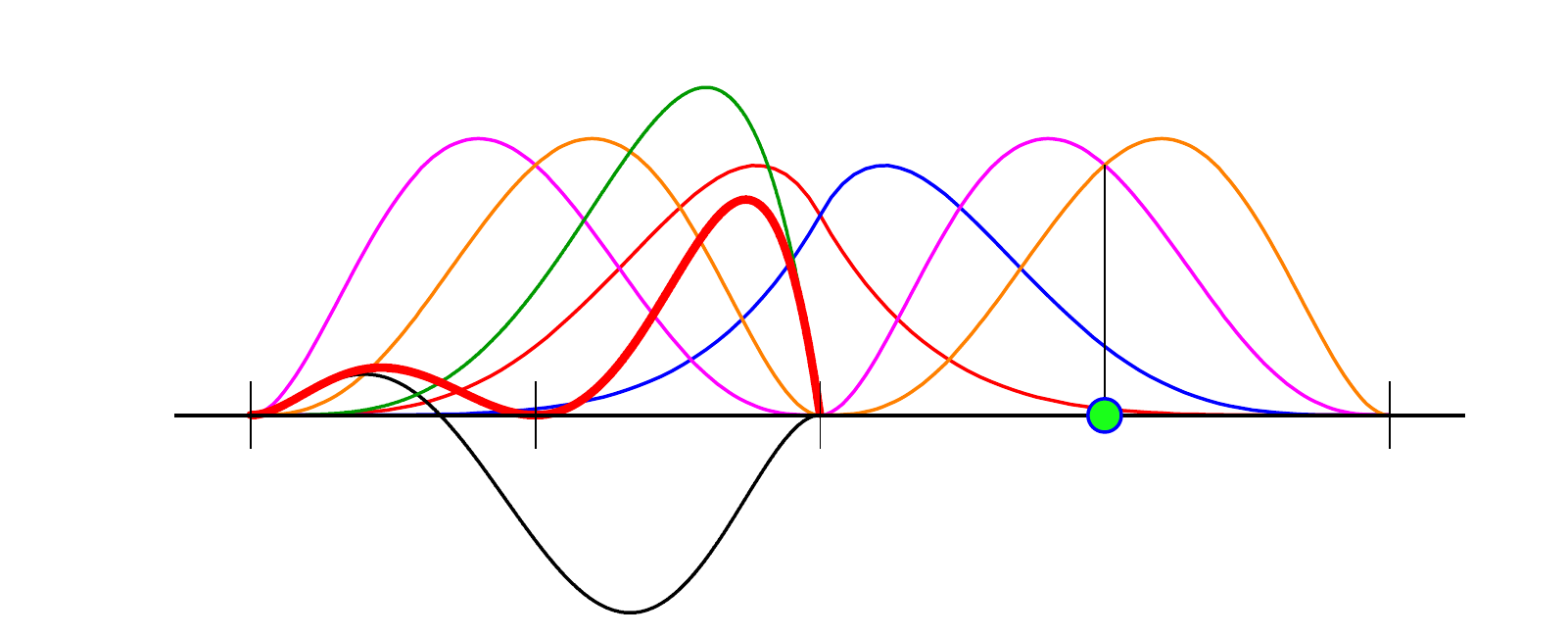}
    \put(15,10){$x_{k-1}$}
    \put(34,10){$\frac{x_{k-1}+x_{k}}{2}$}
    \put(51,10){$x_{k}$}
    \put(88,10){$x_{k+1}$}
    \put(70,10){$\tau_i$}
    \put(52,18){\color{red}{$P_k$}}
    \put(45,36){\color{Mgreen}{$2D_{4k+1}-2D_{4k+2}$}}
    \put(55,31){\color{blue}{$D_{4k+2}$}}
    \put(20,32){\color{magenta}{$D_{4k-1}$}}
    \put(32,33){\color{Morange}{$D_{4k}$}}
    \put(82,25){\color{Morange}{$D_{4k+4}$}}
    \put(47,4){$\frac{1}{2}D_{4k-1}-D_{4k}$}
  \end{overpic}
 \hfill
 \vspace{-4pt}
 \Acaption{1ex}{The assumption of existence of a single node $\tau_i$ inside $[x_{k},x_{k+1}]$ implies $\tau_i = \frac{x_{k+1}+x_{k}}{2}$.
   Consequently, the rule (\ref{quadrature}) must return zero for $P_k$ on $[x_{k-1},x_{k}]$, i.e. $\Q_{x_{k-1}}^{x_{k}}(P_k)=0$.
   As $P_k$ is non-negative on $(x_{k-1},x_{k})$ with one double root $\frac{x_{k-1}+x_{k}}{2}$, this fact violates the assumption of
   a single node in $(x_{k},x_{k+1})$.}\label{fig:Lemma2}
 \end{figure*}
%
\begin{lemma}\label{lemmag1}
Let $\mathcal{X}_{n} = (a=x_0,x_1,...,x_n=b)$ be a uniform knot sequence.
Each of the intervals $J_{k} = (x_{k-1},x_{k})\; (k=1,\dots,n)$ contains at least two nodes of
the Gaussian quadrature rule (\ref{quadrature}).
\end{lemma}
\begin{proof}
We proceed by induction on the index of the segment $J_{k}$.
There must be at least two nodes of the Gaussian quadrature rule inside the interval $J_{1}$.
If there were no node inside $J_{1}$, the exactness of the rule would be violated for $D_{1}$.
If there was
only one node, using the exactness of the quadrature rule for $D_3$ and $D_4$,
it must have been the midpoint $\tau_1 = \frac{(x_{0}+x_{1})}{2}$
with the weight $\omega_1 = \frac{8}{15}h$. However, this contradicts exactness of $D_1$ and $D_2$
as $D_1 < D_2$ on $(x_0,x_1)$.

Now, let us assume that every segment $J_{k}$, $k<n$, contains--two or more--Gaussian nodes
and prove that $J_{k+1}$ contains at least two nodes too. By contradiction,
if there is no node inside $(x_{k},x_{k+1})$, the exactness of the quadrature rule (\ref{quadrature}) for $D_{4k+3}$ is violated.
If there is a single node in $(x_{k},x_{k+1})$, due to the exactness of the quadrature rule (\ref{quadrature}) 
for $D_{4k+3}$ and $D_{4k+4}$,
it must be the midpoint $\tau_i = \frac{(x_{k}+x_{k+1})}{2}$ as it is their only intersection point,
see Fig.~\ref{fig:Lemma2}, and their integrals are equal, see (\ref{interiorIntegral}), $I[D_{4k+3}]=I[D_{4k+4}]=\frac{1}{6}$.
The corresponding weight must be $\omega_i = \frac{8}{15}h$.
Moreover $D_{4k+1}(\tau_i) - D_{4k+2}(\tau_i) = - \frac{5}{64h}$ and combining with $\omega_i$, we have
\begin{equation}\label{residuum}
2\omega_i (D_{4k+1}(\tau_i) - D_{4k+2}(\tau_i)) = -\frac{1}{12}.
\end{equation}
Consider a blend $P_k(t) = 2 D_{4k+1}(t) - 2 D_{4k+2}(t) + \frac{1}{2} D_{4k-1}(t) - D_{4k}(t)$,
see Fig.~\ref{fig:Lemma2}. As $P_k$ is a blend of basis functions, the rule (\ref{quadrature})
must integrate it exactly on $[x_{k-1},x_{k+1}]$, that is $\Q_{x_{k-1}}^{x_{k+1}}(P_k) = I(P_k) =  -\frac{1}{12}$. However, combining this fact
with (\ref{residuum}), the rule must return zero when applied
to $P_k$ on $[x_{k-1},x_{k}]$, i.e. $\Q_{x_{k-1}}^{x_{k}}(P_k)=0$. But due to Lemma~\ref{lemmaBlend}, $P_k$ is non-negative
on $(x_{k-1},x_{k})$ with the only root at $\frac{x_{k-1}+x_{k}}{2}$, which contradicts
the assumption of a single quadrature node in $(x_{k},x_{k+1})$ and completes the proof.
\end{proof}
 \begin{figure*}[!tb]
 \hfill
 \begin{overpic}[width=.8\columnwidth]{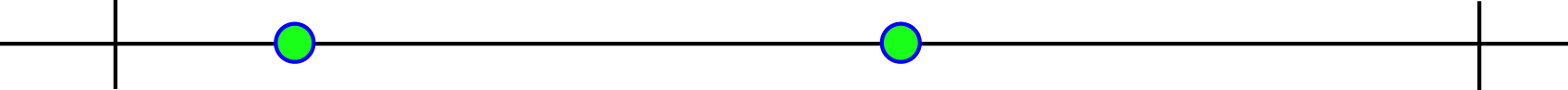}
    \put(2,-1){$x_{k-1}$}
    \put(95,-1){$x_{k}$}
    \put(18,-1){$\tau_{2k-1}$}
    \put(8,5){$\overbrace{\color{white}{aaaaaa}}$}
    \put(57,5){$\overbrace{\color{white}{aaaaaaaaaaaaaaaaaaa}}$}
    \put(55,-1){$\tau_{2k}$}
    \put(51,-9){$h$}
    \put(14,10){$\alpha_k$}
    \put(72,10){$\beta_k$}
    \put(8,-3.5){$\underbrace{\color{white}{aaaaaaaaaaaaaaaaaaaaaaaaaaaaaaaaaaaaaaaaaaaaa}}$}
  \end{overpic}
 \hfill
 \vspace{+20pt}
 \Acaption{1ex}{Notation on $[x_{k-1},x_{k}]$.}\label{fig:Notation}
 \end{figure*}
 %
\begin{cor}\label{corollary1}
If $n$ is an \emph{even} integer, then each of the intervals $J_{k} = (x_{k-1},x_{k})$ $(k=1,2,\ldots,n)$
contains exactly two Gaussian nodes and the middle $x_{n/2} = (a+b)/2$ of the interval $[a,b]$
is also a Gaussian node. If $n$ is \emph{odd} then each of the intervals $J_{k} = (x_{k-1},x_{k})$
$(k=1,2,\ldots,n; k\not= (n+1)/2)$ contains exactly two Gaussian nodes, while the interval
$J_{(n+1)/2}$ contains three Gaussian nodes: the middle $(a+b)/2$ and the other two
positioned symmetrically with respect to $(a+b)/2$.
\end{cor}
\begin{proof}
From \cite{Micchelli-1977}, the optimal quadrature rule (\ref{quadrature}) is known to require $2n+1$ Gaussian nodes.
From Lemma~\ref{lemmag1}, we know the location of $2n$ of them as each of the intervals $J_k$ contains at least two nodes.
The last node must be the midpoint $(a+b)/2$. We prove it by contradiction, distinguishing two cases depending on the parity of $n$.
For $n$ \textit{even}, if one of the intervals $J_{k}$ has more than
two nodes then, by symmetry, $J_{n-k}$ has to contain the same number of nodes and we exceed $2n+1$,
contradicting our quadrature rule (\ref{quadrature}).
For $n$ \textit{odd}, let us assume that the middle
interval $J_{(n+1)/2}$ contains exactly two nodes. Then, by symmetry, at least two of the remaining intervals
contain three nodes, contradicting our quadrature rule (\ref{quadrature}).
Therefore, the middle interval $J_{(n+1)/2}$ contains exactly three nodes, where the middle one is,
again by symmetry, forced to be the midpoint $(a+b)/2$.
\end{proof}

With the knowledge of the exact layout of the optimal quadrature nodes, we now construct a scheme that starts at the boundary of the interval
and parses to its middle, recursively computing the nodes and weights. This process requires to solve only for the
roots of a quadratic polynomial.
Let us denote
\begin{equation}\label{eq:alpha}
\alpha_k = \tau_{2k-1} - x_{k-1}, \quad \beta_k = x_{k} - \tau_{2k},
\end{equation}
where $\tau_{2k-1}$ and $\tau_{2k}$, $\tau_{2k-1} < \tau_{2k}$, are the two quadrature
nodes on $(x_{k-1}, x_k)$, $k=1,\dots, [n/2]+1$, see Fig.~\ref{fig:Notation}. Keeping in mind $h = x_k-x_{k-1}$, we have
\begin{equation}\label{eq:tau}
x_k - \tau_{2k-1} = h-\alpha_k, \quad \tau_{2k} - x_{k-1} = h - \beta_k.
\end{equation}
Let $\omega_{2k-1}$ and $\omega_{2k}$ be
the corresponding weights of the Gaussian quadrature rule over the interval $(x_{k-1}, x_k)$.
The exactness requirement of the rule when applied to $D_{4k-1}$ and $D_{4k}$,
see (\ref{basis2b}) and (\ref{interiorIntegral}), gives the following algebraic constraints
\begin{equation}\label{system1a}
\begin{split}
\omega_{2k-1} \alpha_k^2 (h-\alpha_k)^3 + \omega_{2k} (h-\beta_k)^2 \beta_k^3 & = \frac{h^6}{60}, \\
\omega_{2k-1} \alpha_k^3 (h-\alpha_k)^2 + \omega_{2k} (h-\beta_k)^3 \beta_k^2 & = \frac{h^6}{60}.
\end{split}
\end{equation}
The exactness of the rule when applied on $D_{4k-3}$ and $D_{4k-2}$, respectively, gives
\begin{equation}\label{system1b}
\begin{split}
\omega_{2k-1}  (h-\alpha_k)^5 + \omega_{2k} \beta_k^5  & = 4h^6 r_{4k-3}, \\
\omega_{2k-1} \left(\frac{5(h-\alpha_k)^4}{2h^5} - \frac{9(h-\alpha_k)^5}{4h^6} \right) +
 \omega_{2k}  \left(\frac{5\beta_k^4}{2h^5} - \frac{9\beta_k^5}{4h^6} \right) & = r_{4k-2},
\end{split}
\end{equation}
where $r_{4k-3}$ and $r_{4k-2}$ are the \textit{residues} between the exact integrals,
see (\ref{interiorIntegral}) and (\ref{boundaryIntegral}), and the result of the rule
when applied to $D_{4k-3}$ and $D_{4k-2}$ on the previous interval $[x_{k-2},x_{k-1}]$, respectively.
That is
\begin{equation}\label{eq:residues}
\begin{split}
r_{4k-3} & = I[D_{4k-3}] - \Q_{x_{k-2}}^{x_{k-1}}(D_{4k-3}), \\
r_{4k-2} & = I[D_{4k-2}] - \Q_{x_{k-2}}^{x_{k-1}}(D_{4k-2}).
\end{split}
\end{equation}

Due to the fact that both (\ref{system1a}) and (\ref{system1b}) are linear in $\omega_{2k-1}$ and $\omega_{2k}$,
their elimination from (\ref{system1a}) gives
\begin{equation}\label{eq:elim1a}
\begin{split}
\omega_{2k-1} & = \frac{h^5 (h - 2\beta_k)}{60 \alpha_k^2 (h-\alpha_k)^2 (h - \alpha_k - \beta_k)}, \\
\omega_{2k} & = \frac{h^5 (h - 2\alpha_k)}{60 \beta_k^2 (h-\beta_k)^2 (h - \alpha_k - \beta_k)},
\end{split}
\end{equation}
and from (\ref{system1b}) we obtain
\begin{equation}\label{eq:elim1b}
\begin{split}
 \omega_{2k-1} &  = \frac{-2h^5(9\beta_kr_{4k-3}- 10hr_{4k-3} + \beta_kr_{4k-2})}
{5 (h-\alpha_k)^4 (h - \alpha_k - \beta_k)}, \\
 \omega_{2k} & = \frac{-2h^5(hr_{4k-3} + \alpha_kr_{4k-2} + 9r_{4k-3}\alpha_k - hr_{4k-2})}
{5 \beta_k^4 (h - \alpha_k - \beta_k)}.
\end{split}
\end{equation}
Equating $\omega_{2k-1}$ and $\omega_{2k}$ from (\ref{eq:elim1a}) and (\ref{eq:elim1b})
we obtain
\begin{equation}\label{eq:fg}
\begin{split}
& \Phi_k(\alpha_k, \beta_k ) =  0,\\
& \Psi_k(\alpha_k, \beta_k ) =  0,
\end{split}
\end{equation}
an algebraic system of degree three with the unknowns $\alpha_k$ and $\beta_k$. Solving this
non-linear system of two equations with two unknowns can
be interpreted as the intersection problem of two algebraic curves, see Fig.~\ref{fig:AlgSys}.
The domain of interest is $(0,h)\times (0,h)$ as both quadrature points lie inside $(x_{k-1},x_k)$.

 \begin{figure*}
 \hfill
 \begin{overpic}[width=.4\columnwidth]{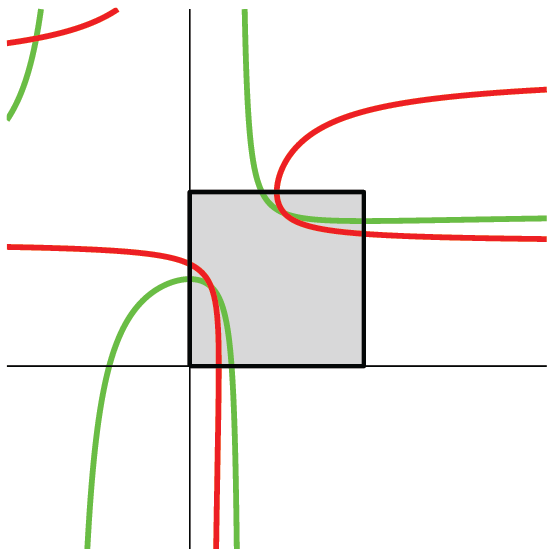}
    \put(26,90){$\beta_{1}$}
    \put(29,27){$0$}
    \put(65,27){$h$}
    \put(90,27){$\alpha_1$}
    \put(55,70){\color{red}{$\Phi_1(\alpha_1,\beta_1)=0$}}
    \put(45,5){\color{Mgreen}{$\Psi_1(\alpha_1,\beta_1)=0$}}
    \put(34,45.5){\huge $*$}
    \put(46.5,58.5){\huge $*$}
  \end{overpic}
 \hfill
 \vspace{-4pt}
 \Acaption{1ex}{The algebraic system (\ref{eq:fg}) over the domain $(0,h)\times (0,h)$ (grey) for the first ($k=1$) subinterval $[x_{0},x_1]$
 is shown. The two intersection points
 correspond to the two Gaussian nodes on $[x_{0},x_1]$ and are computed by projection onto $\alpha_1$-axis using the resultant.
 The coordinates of the intersection points with respect to $\alpha_1$-axis are the roots of the quadratic polynomial (\ref{eq:quadratic}).
 }\label{fig:AlgSys}
 \end{figure*}

Using the resultant, see e.g. \cite{Cox-2005}, of these two algebraic curves in the direction of $\beta_k$,
one obtains a univariate polynomial, in general, of degree nine. Interestingly,
our system (\ref{eq:fg}) produces--for all admissible residues $r_{4k-3}$ and $r_{4k-2}$--only 
a quintic polynomial $E_k(\alpha_k)$ that gets factorized over $\mathbb R$ as
\begin{equation}\label{eq:e5}
Res_{\beta_k}(\Phi_k(\alpha_k, \beta_k ), \Psi_k(\alpha_k, \beta_k )) = E_k(\alpha_k) = Q_k(\alpha_k) C_k(\alpha_k),
\end{equation}
where $Q_k$
is a quadratic factor and the vector of its coefficients with respect to the \textit{monomial} basis, $\bq_{mo}^k = (q^k_0,q^k_1,q^k_2)$,
is--in terms of the residues--expressed as
%
\begin{equation}\label{eq:quadratic}
\begin{split}
& q^k_2 = 1-480r_{4k-3} + 576r_{4k-3}^2 + 576r_{4k-2}^2 - 1152r_{4k-2}r_{4k-3},\\
& q^k_1 = 2h(12r_{4k-2} + 108r_{4k-3} - 1),\\
& q^k_0 = h^2(1-24r_{4k-2} + 24r_{4k-3}),
\end{split}
\end{equation}
and for the vector of monomial coefficients $\bc_{mo}^k = (c_0^k,c_1^k,c_2^k,c_3^k)$ of the cubic factor $C_k$ we obtain
\begin{equation}\label{eq:cubic}
\begin{split}
& c^k_3 = -216r_{4k-3} - 24r_{4k-2} + 2,\\
& c^k_2 = h(24r_{4k-2} -24r_{4k-3} - 5),\\
& c^k_1 = 4h^2,\\
& c^k_0 = -h^3.
\end{split}
\end{equation}
A Maple worksheet with this algebraic factorization is attached to this submission.

We recall that two roots of $E_k$ (\ref{eq:e5}) determine the two quadrature nodes that lie inside $[x_{k-1},x_k]$.
Interestingly, the cubic factor does not contribute to the computation of
the nodes which is formalized as follows.
%
\begin{lemma}\label{lem:NoRoot}
The cubic polynomial $C_k$ (\ref{eq:cubic}) has no roots inside $[0,h]$.
\end{lemma}
A proof can be found in Appendix.

\medskip
We now proceed to the main contribution of the paper,
a recursive algorithm that computes the nodes and the weights
of the Gaussian quadrature for uniform $C^1$ quintic splines. Due to Lemma~\ref{lem:NoRoot},
the recursion operates in a closed form fashion by solving for the roots of quadratic polynomial (\ref{eq:quadratic}).
Before we state the theorem, we need to establish some notation.

 \begin{figure*}[!tb]
 \hfill
 \begin{overpic}[width=.91\columnwidth]{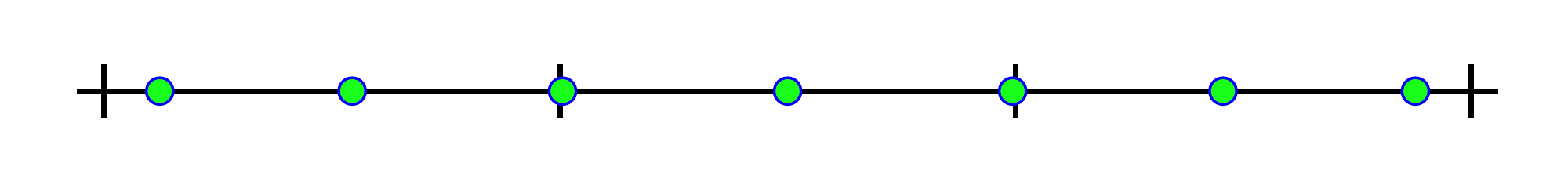}
    \put(50,3){$\tau_{n+1}$}
    \put(63,3){$\tau_{n+2}$}
    \put(36,3){$\tau_{n}$}
    \put(5,10){$x_{m-2}$}
    \put(35,10){$x_{m-1}$}
    \put(64,10){$x_{m}$}
    \put(93,10){$x_{m+1}$}
  \end{overpic}
 \hfill
 \vspace{0pt}
 \Acaption{1ex}{The situation for odd $n$ on the middle interval $J_m = [x_{m-1},x_m]$. The node $\tau_{n+1}$ is the middle
  of the interval, $\tau_n$ and $\tau_{n+2}$ are computed from (\ref{eq:MainOddNs}).}\label{fig:MiddleOdd}
 \end{figure*}

Let us denote by $A_{k}$ and $B_k$, $k=2,\dots, [n/2]+1$, the actual values of residues (\ref{eq:residues}) when being evaluated at
the nodes $\tau_{2k-3}$ and $\tau_{2k-2}$ on the interval $[x_{k-2},x_{k-1}]$, i.e.,
\begin{equation}\label{eq:evalAB}
\begin{split}
A_{k} & = I[D_{4k-3}] - \omega_{2k-3} D_{4k-3}(\tau_{2k-3}) - \omega_{2k-2} D_{4k-3}(\tau_{2k-2}), \\
B_{k} & = I[D_{4k-2}] - \omega_{2k-3} D_{4k-2}(\tau_{2k-3}) - \omega_{2k-2} D_{4k-2}(\tau_{2k-2}),
\end{split}
\end{equation}
and the coefficients of the quadratic polynomial (\ref{eq:quadratic}) become
\begin{equation}\label{eq:evalQuadratic}
\begin{split}
& a_k = 1-480A_{k} + 576A_{k}^2 + 576B_{k}^2 - 1152B_{k}A_{k},\\
& b_k = 2h(12B_{k} + 108A_{k} - 1),\\
& c_k = h^2(1-24B_{k} + 24A_{k}).
\end{split}
\end{equation}
In the case when $n$ is odd, see Fig.~\ref{fig:MiddleOdd}, the middle subinterval contains three nodes
and the algebraic system that needs to be solved results in a quadratic polynomial
with the following coefficients
\begin{equation}\label{eq:evalMiddleOdd}
\begin{split}
& \widetilde{a}_m = -2(108A_{m} + 12B_m + 1),\\
& \widetilde{b}_m = 2h(108A_{m} + 12B_m - 1),\\
& \widetilde{c}_m = h^2(12A_m -12B_m +1).
\end{split}
\end{equation}
We are now ready to formalize the main theorem.
\begin{theorem}\label{theo:Quadrature}
%
%
The sequences of nodes and weights of the Gaussian quadrature rule (\ref{quadrature})
are given explicitly by the initial values $A_1=\frac{1}{24}$ and $B_1=\frac{1}{8}$
and the recurrence relations $(k =1,...,[n/2])$ for the nodes
%
%
\begin{equation}\label{eq:MainQuadratureNs}
\renewcommand{\arraystretch}{1.4}
\begin{array}{c}
\tau_{2k-1} = x_{k-1} + \frac{-b_k-\sqrt{b_k^2-4a_kc_k}}{2a_k} \quad \text{and} \quad
\tau_{2k} = x_{k} - \frac{-b_k+\sqrt{b_k^2-4a_kc_k}}{2a_k}
\end{array}
\end{equation}
and for the weights
\begin{equation}\label{eq:MainQuadratureWs}
\begin{split}
\omega_{2k-1} & = \frac{h^5 (2\tau_{2k} - 2x_{k-1} - h)}{60 (\tau_{2k-1}-x_k+h)^2 (x_k-\tau_{2k-1})^2 (\tau_{2k} - \tau_{2k-1})}, \\
\omega_{2k}   & = \frac{h^5 (2x_{k} - 2\tau_{2k-1} - h)}{60 (x_{k-1}-\tau_{2k}+h)^2 (\tau_{2k} - x_{k-1})^2 (\tau_{2k} - \tau_{2k-1})}.
\end{split}
\end{equation}
If $n$  is even ($n=2m$) then $\tau_{n+1} = x_{m} = (a+b)/2$ and
\begin{equation}\label{eq:MainEven}
\omega_{n+1} = 4h(\frac{1}{6} - A_{m+1}).
\end{equation}
If $n$ is odd, ($n=2m-1$) then $\tau_{n+1} = (a+b)/2$,
\begin{equation}\label{eq:MainOddNs}
\begin{array}{c}
\tau_{n} = x_{m-1} + \frac{-\widetilde{b}_m-\sqrt{\widetilde{b}_m^2-4\widetilde{a}_m\widetilde{c}_m}}{2\widetilde{a}_m}
\quad \text{and} \quad
\tau_{n+2} = x_{m} - \frac{-\widetilde{b}_m+\sqrt{\widetilde{b}_m^2-4\widetilde{a}_m\widetilde{c}_m}}{2\widetilde{a}_m}
\end{array}
\end{equation}
and the corresponding weights are
\begin{equation}\label{eq:MainOddWs}
\begin{split}
\omega_{n}   & =  \omega_{n+2} = \frac{h(108A_{m}+12B_{m}-1)^2}{30(156A_{m}-36B_{m}+1)},\\
\omega_{n+1} & =  \frac{4h(1152A_{m}B_{m}+264A_{m}-576A_{m}^2-576B_{m}^2-24B_{m}+1)}{15(156A_{m}-36B_{m}+1)}.\\
\end{split}
\end{equation}
\end{theorem}
\begin{proof}
We proceed by induction.
Assume the quadrature nodes $(\tau_{2l-1}, \tau_{2l})$ and weights $(\omega_{2l-1},\omega_{2l})$
are known for $l=1,\ldots,k-1$ ($k \leq [n/2]$) and compute the new ones on $(x_{k-1}, x_k)$.
For $k=1$, as there are no nodes on $(x_{-1}, x_0)$,
(\ref{eq:evalAB}) gives $A_1 = I[D_{1}] = \frac{1}{24}$ and $B_1 = I[D_{2}] = \frac{1}{8}$.
By Corollary~\ref{corollary1}, there are exactly two nodes inside $(x_{k-1},x_{k})$.
Due to Lemma~\ref{lem:NoRoot}, only the roots of the quadratic factor in (\ref{eq:e5})
contribute to the computation of the nodes and hence solving $Q_k(\alpha_k)=0$ with coefficients from (\ref{eq:evalQuadratic})
gives $\alpha_k$ and $\beta_k$. Combining these with (\ref{eq:alpha})
results in (\ref{eq:MainQuadratureNs}). The weights are computed from (\ref{eq:elim1a})
using the identities (\ref{eq:alpha}) and (\ref{eq:tau}).
By Corollary~\ref{corollary1}, we have $\tau_{n+1} = (a+b)/2$.
If $n$ is even, using the exactness of the quadrature for $D_{2n+1}$, the associated weight is computed
from
\begin{equation}
\frac{1}{6} = I[D_{2n+1}] = A_{m+1} + \omega_{n+1} D_{2n+1}(\tau_{n+1}).
\end{equation}
Evaluating $D_{2n+1}((a+b)/2) = \frac{1}{4h}$ gives (\ref{eq:MainEven}).
If $n$ is odd, due Corollary~\ref{corollary1}, there are three nodes inside $(x_{m-1},x_m)$; one is the middle point $(a+b)/2$ and the other two
are symmetric with respect to it, see Fig.~\ref{fig:MiddleOdd}.
Using the notation of (\ref{eq:alpha}) for the middle interval, i.e., $\alpha_m = \tau_{2m-1} - x_{m-1}$,
the rule must integrate exactly $D_{4m-3}$, $D_{4m-2}$ and $D_{4m-1}$ which
gives the following $3\times 3$ algebraic system
\begin{equation}\label{eq:MiddleSystem}
\renewcommand{\arraystretch}{1.4}
\begin{array}{rcccl}
\frac{(h-\alpha_m)^5+\alpha_m^5}{4h^6}\omega_{n} & + & \frac{1}{128h}\omega_{n+1} & = & A_m, \\
\frac{(h-\alpha_m)^4(9\alpha_m+h) + \alpha_m^4(10h-9\alpha_m)}{4h^6}\omega_{n} & + & \frac{11}{128h}\omega_{n+1} & = & B_m, \\
\frac{10\alpha_m^2(h-\alpha_m)^2}{h^5}\omega_{n} & + & \frac{5}{16h}\omega_{n+1} & = & \frac{1}{6},
\end{array}
\end{equation}
with unknowns $\alpha_m$, $\omega_n$ and $\omega_{n+1}$. Eliminating $\omega_{n+1}$ from the first two and
second two equations, respectively, and solving for $\omega_{n}$ we obtain
\begin{equation}\nonumber
\omega_n = \frac{2h^5(11A_m-B_m)}{5(h^2-2h\alpha_m+2\alpha_m^2)(h-2\alpha_m)^2}
=\frac{h^5(240B_m-11)}{60(h^2+9h\alpha_m-9\alpha_m^2)(h-2\alpha_m)^2}
\end{equation}
and the problem reduces to solving for the roots of a univariate (rational) function in $\alpha_m$. The numerator
is a quadratic polynomial with coefficients (\ref{eq:evalMiddleOdd}) which proves (\ref{eq:MainOddNs}).
Inserting (\ref{eq:evalMiddleOdd}) into (\ref{eq:MiddleSystem}) and solving for $\omega_n$ and $\omega_{n+1}$
then gives (\ref{eq:MainOddWs}) and completes the proof.
\end{proof}

For the convenience, we summarize the recursion in Algorithm~\ref{algor}.

\begin{algorithm}[!tb]\normalsize \caption{\hfill { {\tt
GaussianQuadrature}$([a,b],n)$} }\label{algor}
\begin{algorithmic}[1]
\medskip
\STATE \textbf{INPUT}: compact interval $[a,b]$ and
number of uniform segments $n$\\
\vspace{-0.35cm}
 \begin{tabular}{c}
  \makebox[11cm]{}\\ \hline
 \end{tabular}
\vspace{0.15cm}
\STATE $A_1=\frac{1}{24}$; $B_1=\frac{1}{8}$;
\FOR{$k=1$ to $[n/2]$}
\STATE compute $\tau_{2k-1}$, $\tau_{2k}$ from (\ref{eq:MainQuadratureNs}),
and $\omega_{2k-1}$, $\omega_{2k}$ from (\ref{eq:MainQuadratureWs});
\ENDFOR
\STATE $\tau_{n+1} = (a+b)/2$; \hfill {\tt /*
middle node */}
\IF{$n$ is even}
\STATE compute $\omega_{n+1}$ from (\ref{eq:MainEven});
\ELSE
\STATE compute $\tau_n$ and $\tau_{n+2}$ from (\ref{eq:MainOddNs}) and $\omega_{n}$, $\omega_{n+1}$ and $\omega_{n+2}$ from (\ref{eq:MainOddWs});
\ENDIF
\FOR{$k=1$ to $[n/2]$} %
\STATE $\tau_{2n-2k+3} = \tau_{2k-1}$; \quad $\tau_{2n-2k+2} = \tau_{2k}$; \hfill {\tt /*symmetry */}
\STATE $\omega_{2n-2k+3} = \omega_{2k-1}$; \quad $\omega_{2n-2k+2} = \omega_{2k}$;
\ENDFOR
\STATE \textbf{OUTPUT}: $\{\tau_i,\omega_i\}_{i=1}^{2n+1}$, set of nodes and weights
of the Gaussian quadrature on interval $[a,b]$;
\end{algorithmic} \end{algorithm}

\section{Error estimation for the $C^1$ quintic splines quadrature rule}\label{sec:err}

In the previous section, we have derived a quadrature rule that exactly
integrates functions from $S^{n}_{5,1}$ with uniform knot sequences.
If $f$ is not an element of $S^{n}_{5,1}$, the rule produces a certain error, also known as \emph{remainder}.
The analysis of this error is the objective of this section.

Let $W_{1}^{r} = \{ f \in C^{r-1}[a,b]; \; f^{(r-1)} \textnormal{abs. cont.}, \;   ||f^{(r)}||_{L_1} < \infty \}$.
As the quadrature rule (\ref{quadrature}) is exact for polynomials of degree at most five,
for any element $f \in W_{1}^{d}$, $d \geq 6$, we have
\begin{equation*}
R_{2n+1}[f] :=  I[f] - \Q_a^b[f] = \int_{a}^{b} K_6(R_{2n+1};t) f^{(6)}(t) dt,
\end{equation*}
where the Peano kernel \cite{Gautschi-1997} is given by
\begin{equation*}
 K_6(R_{2n+1};t) = R_{2n+1} \left[ \frac{(t-.)_{+}^{5}}{5!} \right].
\end{equation*}
An explicit representation for the Peano kernel over the interval $[a,b]$ in terms of the
weights and nodes of the quadrature rule (\ref{quadrature})  is given by
\begin{equation}\label{eq:PeanoKernel}
 K_6(R_{2n+1};t) = \frac{(t-a)^6}{720} - \frac{1}{120} \sum_{k=1}^{2n+1} \omega_k (t - \tau_k)_{+}^{5}.
\end{equation}
Moreover, according to a general result for monosplines and quadrature rules \cite{Micchelli-1977}, the only zeros
of the Peano kernel over $(a,b)$ are the knots of multiplicity four of the quintic spline.
Therefore, for any $t \in (a,b)$,
$K_6(R_{2n+1};t) \geq 0$ and, by the mean value theorem for integration, there exists a real number $\xi \in [a,b]$ such that
for $f\in C^6[a,b]$
\begin{equation}\label{remainder}
R_{2n+1}(f) = c_{2n+1,6}  f^{(6)}(\xi) \quad  \textnormal{with } \quad  c_{2n+1,6} =  \int_{a}^{b} K_{6}(R_{2n+1}; t) dt.
\end{equation}
Hence, the constant $c_{2n+1,6}$ of the remainder $R_{2n+1}$ is always positive and our quadrature rule belongs to the
family of positive definite quadratures of order $6$, e.g., see \cite{Nikolov-1996, Nikolov-1995, Schmeisser-1972}.
Integration of (\ref{eq:PeanoKernel}) gives
\begin{theorem}
The error constant $c_{2n+1,6}$ in (\ref{remainder}) of the quadrature rule (\ref{quadrature}) is given by
\begin{equation}\label{eq:c}
c_{2n+1,6} =  \frac{(b-a)^7}{5040} - \frac{1}{720} \sum_{k=1}^{2n+1} \omega_k (\tau_{k} - a)^6.
\end{equation}
\end{theorem}

\section{Numerical Experiments}\label{sec:ex}

\begin{figure}[!tb]
\vrule width0pt\hfill
 \begin{overpic}[width=.49\columnwidth,angle=0]{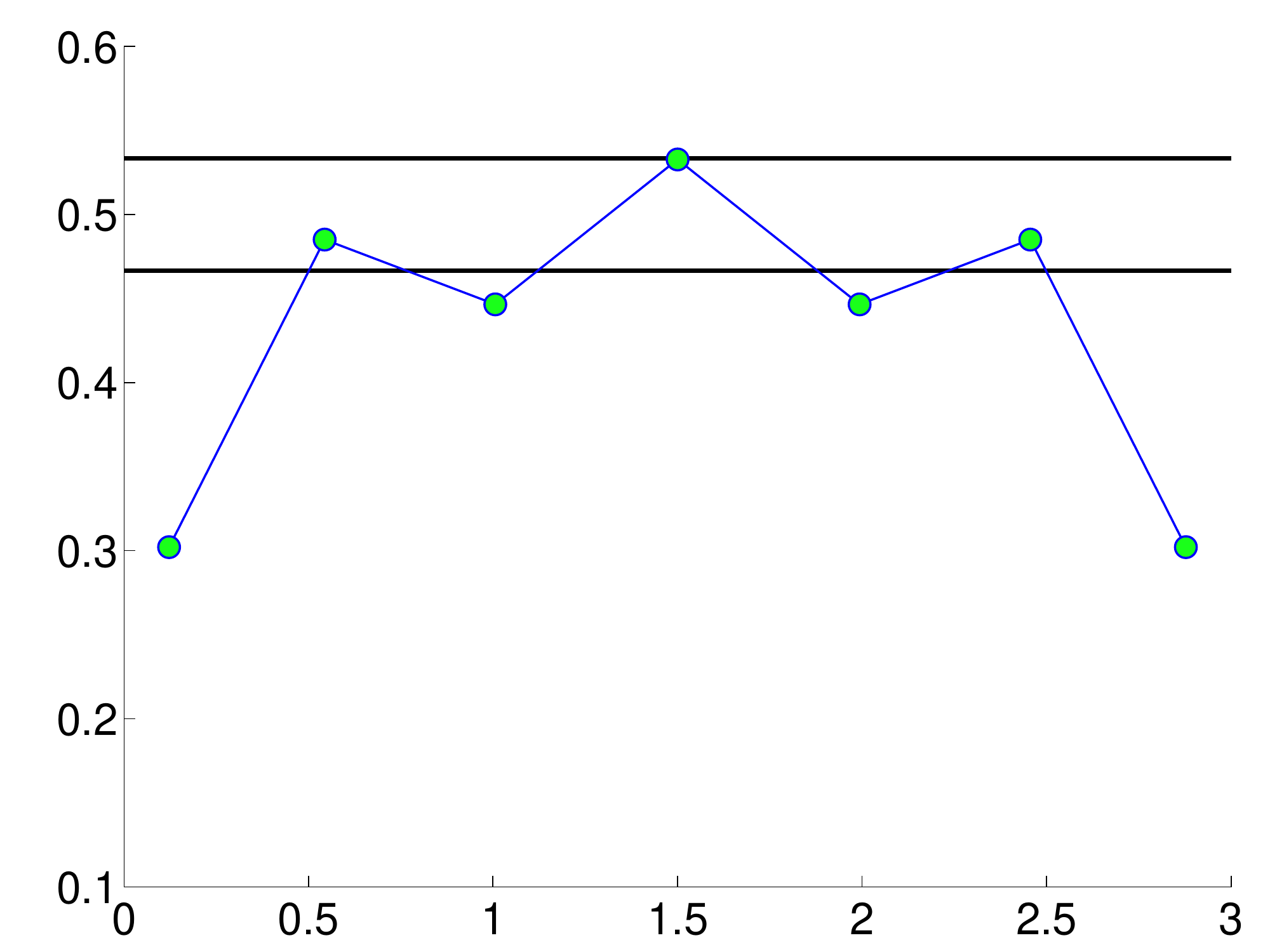}
   \put(50,80){\fcolorbox{gray}{white}{odd}}
   \put(13,68){\fcolorbox{gray}{white}{\small $n=3$}}
	\end{overpic}
 \hfill
 \begin{overpic}[width=.49\columnwidth,angle=0]{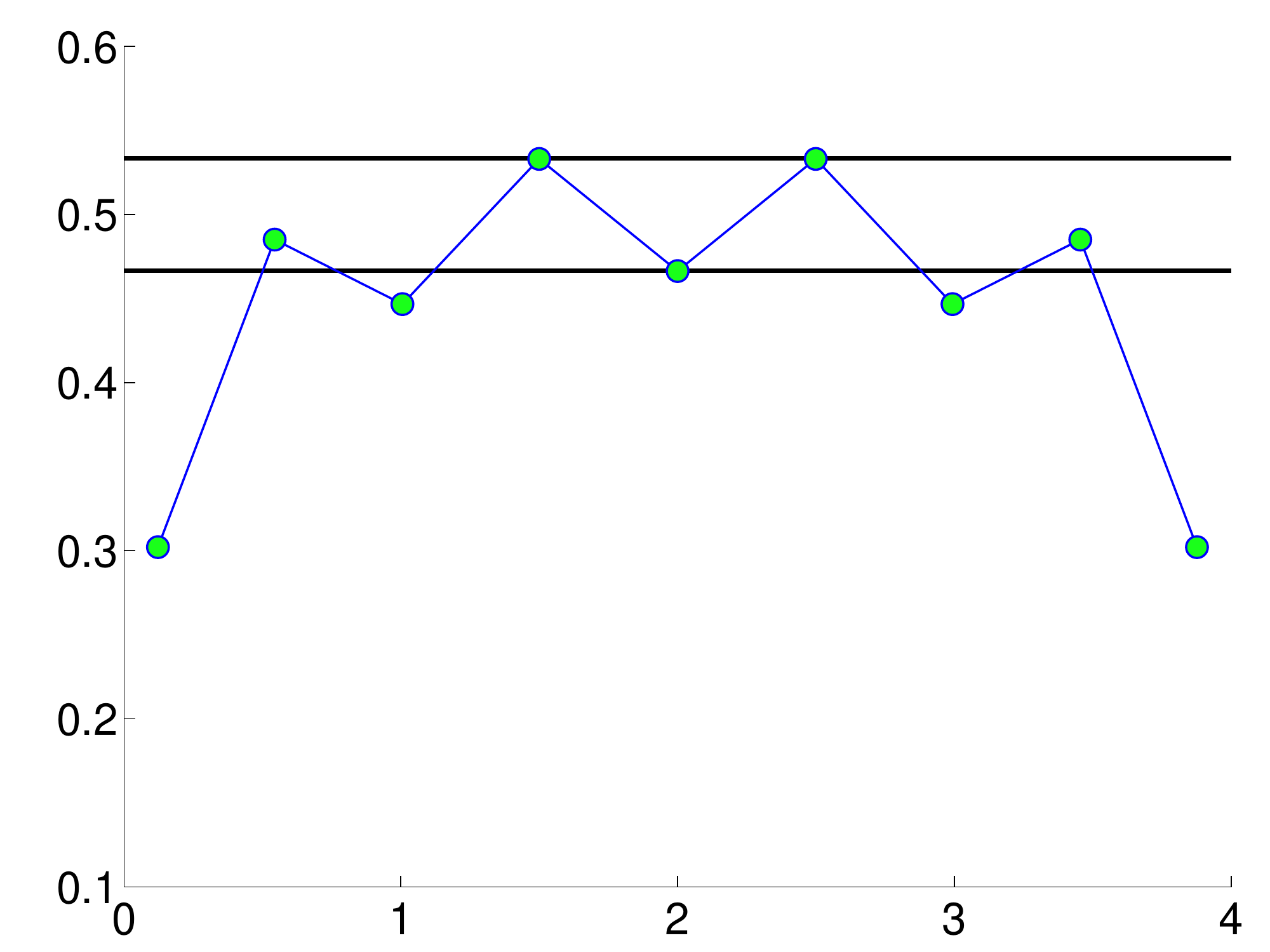}
 \put(50,80){\fcolorbox{gray}{white}{even}}
 \put(13,68){\fcolorbox{gray}{white}{\small $n=4$}}
    \put(25,45){$[\tau_3,\omega_3]$}
    \put(35,66){$[\tau_4,\omega_4]$}
	\end{overpic}\hfill \vrule width0pt\\[1ex]
\vrule width0pt\hfill
 \begin{overpic}[width=.49\columnwidth,angle=0]{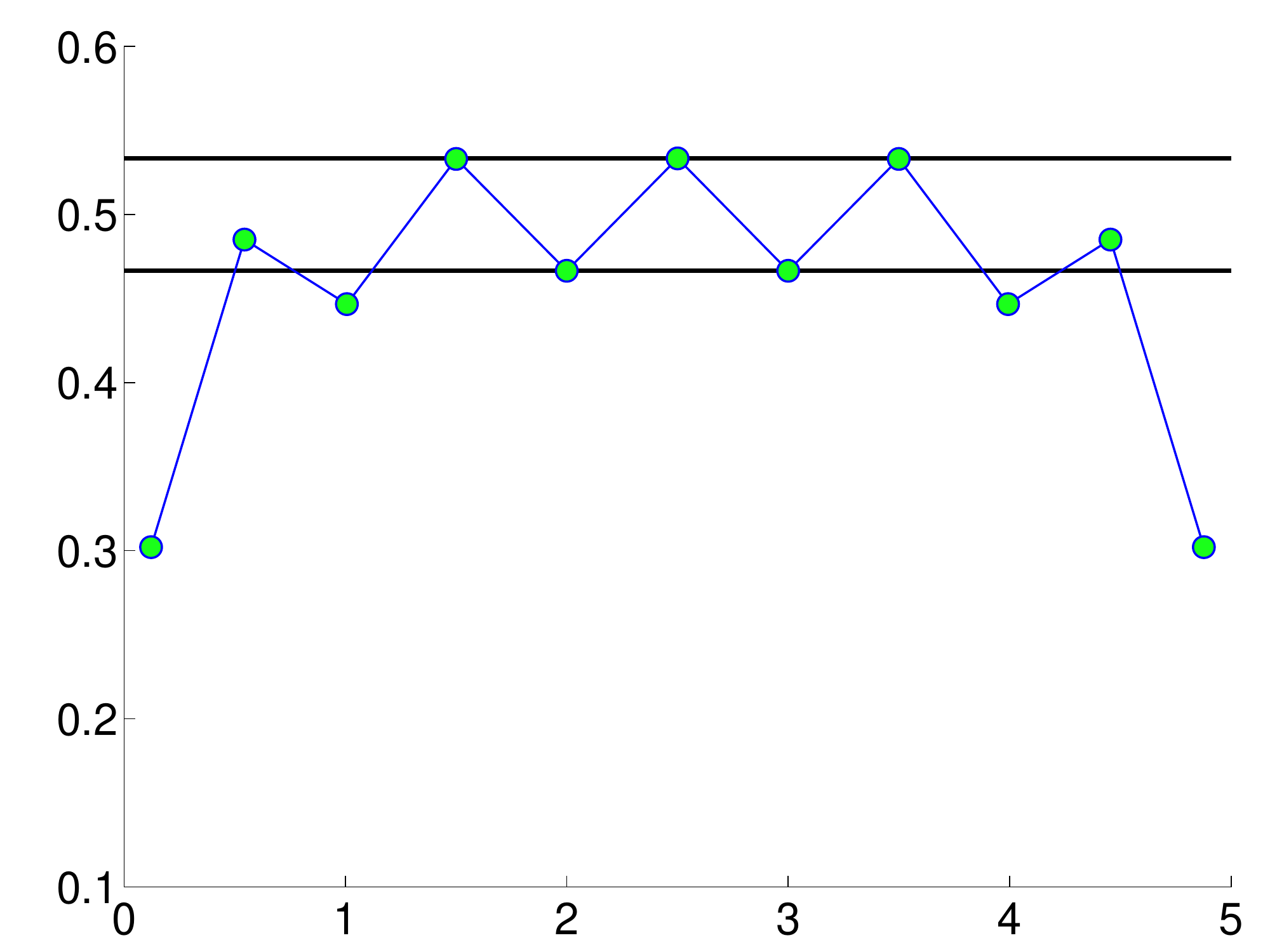}
 \put(13,68){\fcolorbox{gray}{white}{\small $n=5$}}
 \put(55,-10){$\vdots$}
	\end{overpic}
 \hfill
 \begin{overpic}[width=.49\columnwidth,angle=0]{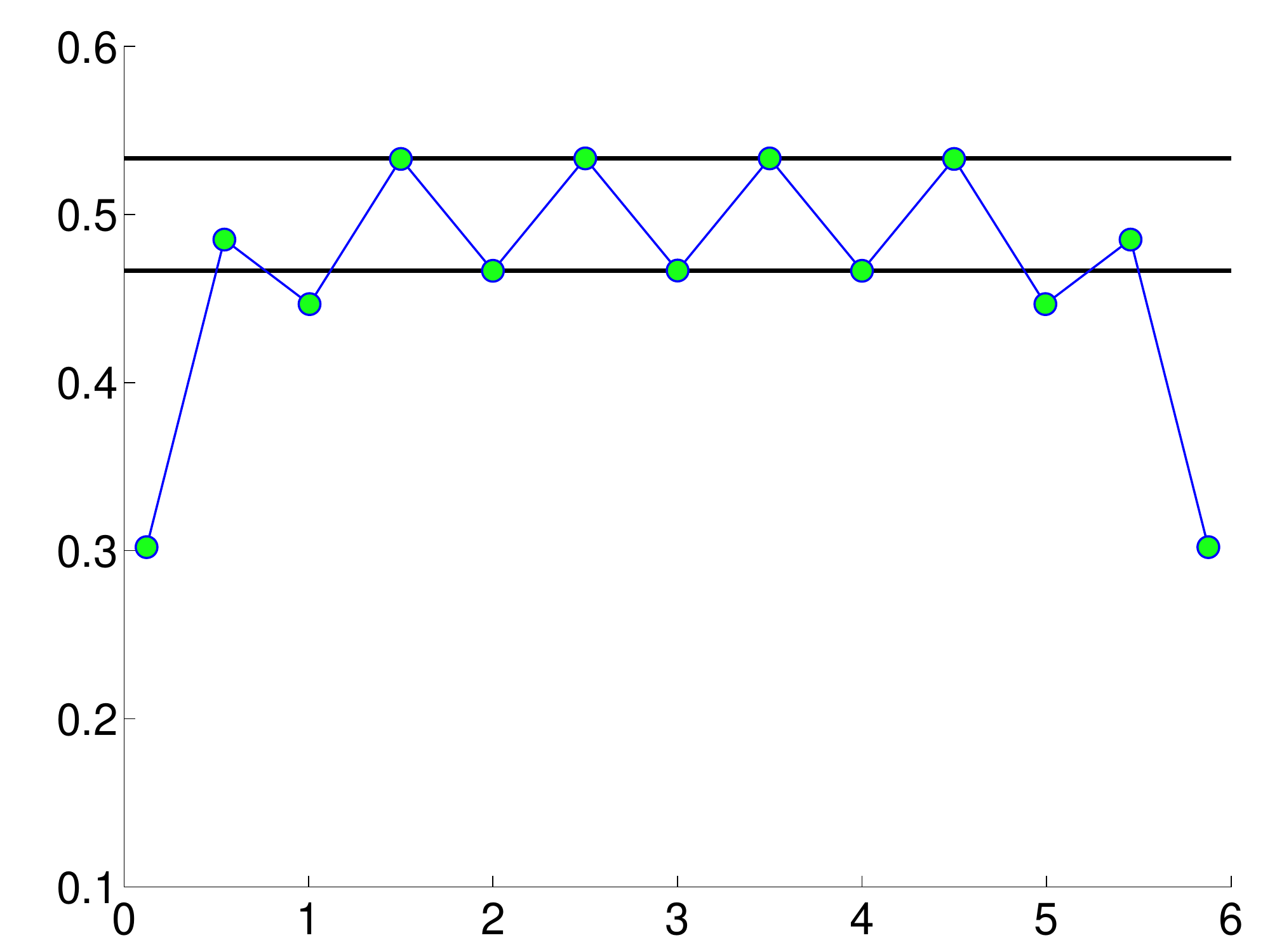}
 \put(13,68){\fcolorbox{gray}{white}{\small $n=6$}}
 \put(55,-10){$\vdots$}
	\end{overpic}\hfill \vrule width0pt\\[4ex]
\vrule width0pt\hfill
 \begin{overpic}[width=.49\columnwidth,angle=0]{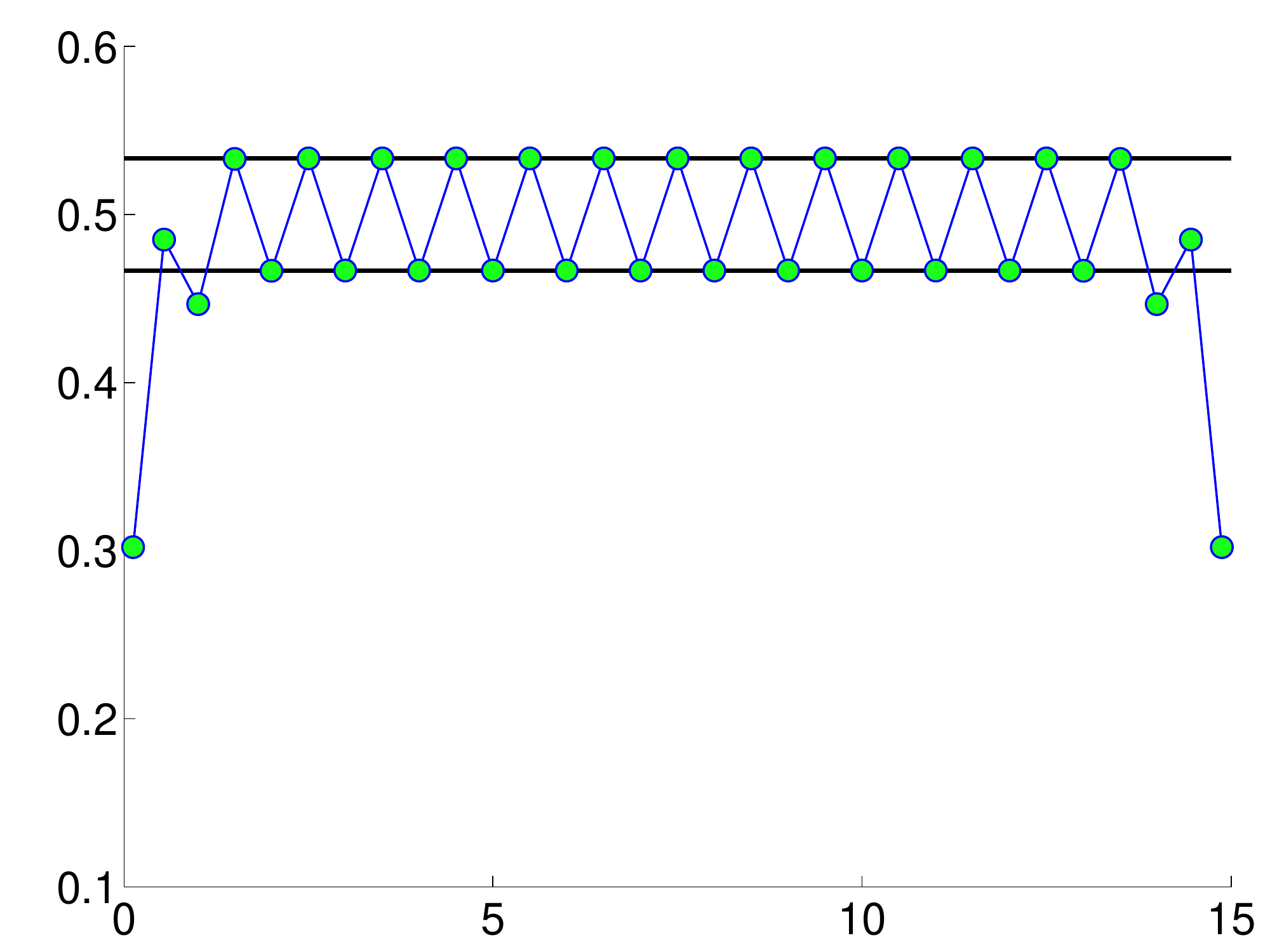}
 \put(13,68){\fcolorbox{gray}{white}{\small $n=15$}}
	\end{overpic}
 \hfill
 \begin{overpic}[width=.49\columnwidth,angle=0]{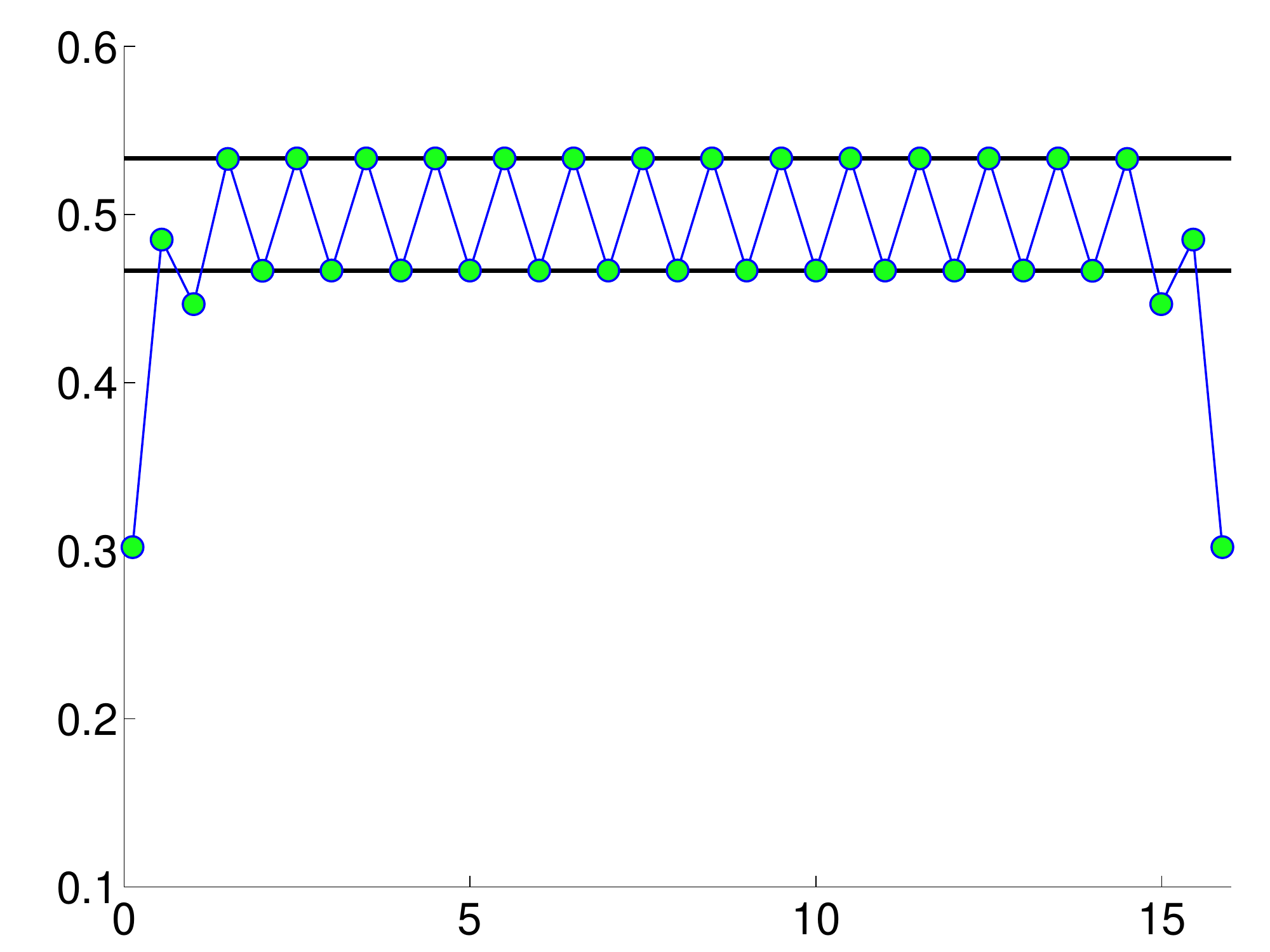}
 \put(13,68){\fcolorbox{gray}{white}{\small $n=16$}}
	\end{overpic}\hfill \vrule width0pt\\
 \vspace{-15pt}
\Acaption{1ex}{The quadrature rules (\ref{quadrature}) for various $n$ are shown.
The interval is set $[a,b] = [0,n]$, i.e., the distance between the neighboring knots is normalized to $h=1$.
The green dots visualize the quadrature; the $x$-coordinates are the nodes and the $y$-coordinates the corresponding weights.
As $n \rightarrow \infty$, the nodes converge to the knots and the midpoints of the subintervals, the weights converge to $0.4\overline{6}$
and $0.5\overline{3}$ (black lines), cf. Table~\ref{tabW} and (\ref{eq:limit}).}\label{fig:BasisNsWs}
 \end{figure}

In this section, we show some examples of quadrature nodes and weights
for particular numbers of subintervals and discuss the asymptotic behaviour
of the rule for $n \rightarrow \infty$.

In the case when the domain is the whole real line, the exact and optimal
rule is easy to compute. Similarly to \cite{Hughes-2010}, Eq.(29), where the rule
was derived for $S_{4,1}$, for $S_{5,1}$ case one obtains
\begin{equation}\label{eq:limit}
\int_{\mathbb R}f(t)\, \mathrm{d}t = \sum_{i \in \mathbb Z} h(\frac{7}{15} f(ih) + \frac{8}{15} f(\frac{2i+1}{2}h)),
\end{equation}
that is, the nodes are the knots and the middles of the subintervals. Similarly to \cite{Hughes-2010},
only two evaluations per subinterval are needed which gives $2/3$ cost reduction ratio when
compared to the classical Gaussian quadrature for polynomials.
Observe the convergence of our general uniform rule
to its limit, (\ref{eq:limit}), when $n\rightarrow \infty$. The weights and nodes are shown in Table~\ref{tabW}.
Only few initial nodes and weights differ from the limit values as $\frac{7}{15}=0.4\overline{6}$
and $\frac{8}{15}=0.5\overline{3}$. From Table~\ref{tabW} we conclude that for large values of $n$,
one needs to compute only the first nine nodes and weights that differ from the limit values by more than $\varepsilon = 10^{-16}$.

\begin{table}[!tb]
 \begin{center}
  \begin{minipage}{0.9\textwidth}
\caption{Nodes and weights for Gaussian quadrature (\ref{quadrature}) with double-precision for various $n$ are shown.
To observe the convergence to (\ref{eq:limit}), the interval was set as $[a,b]=[0,n]$.
Due to the symmetry, only the first $n+1$ nodes and weights are displayed.}\label{tabW}
  \end{minipage}
\vspace{0.2cm}\\
\footnotesize{
\renewcommand{\arraystretch}{1.2}
\begin{tabular}{| c || l| l| l| l| l| l| l|}\hline
\rotatebox{0}{}
 & \multicolumn{2}{c|}{$n=5$}
 & \multicolumn{2}{c|}{$n=6$}\\
 $i$ &  \multicolumn{1}{c|}{$\tau_i$}  & \multicolumn{1}{c|}{$\omega_i$} & \multicolumn{1}{c|}{$\tau_i$} & \multicolumn{1}{c|}{$\omega_i$}
\\\hline\hline
 1 & \tf 0.1225148226554413 & \tf 0.3020174288145723 & \tf 0.1225148226554413 & \tf 0.3020174288145723  \\\hline
 2 & \tf 0.5441518440112252 & \tf 0.4850196082224646 & \tf 0.5441518440112252 & \tf 0.4850196082224646 \\\hline
 3 & \tf 1.0064654716056596 & \tf 0.4467177201362911 & \tf 1.0064654716056596 & \tf 0.4467177201362911 \\\hline
 4 & \tf 1.5002730728687338 & \tf 0.3303872093804185 & \tf 1.5002730728687338 & \tf 0.5330387209380418 \\\hline
 5 & \tf 2.0000387957905171 & \tf 0.4665398664562177 & \tf 2.0000387972956304 & \tf 0.4665398713719121 \\\hline
 6 & \tf 2.5                & \tf 0.5333333108648244 & \tf 2.5000000105321137 & \tf 0.5333333220982075 \\\hline
 7 & --                     & --                     & \tf 3                  & \tf 0.4666666568370204 \\\hline\hline
   & \multicolumn{2}{c|}{$n=7$} & \multicolumn{2}{c|}{$n=8$}\\
 1 & \tf 0.1225148226554413 & \tf 0.3020174288145723 & \tf 0.1225148226554413 & \tf 0.3020174288145723 \\\hline
 2 & \tf 0.5441518440112252 & \tf 0.4850196082224646 & \tf 0.5441518440112252 & \tf 0.4850196082224646 \\\hline
 3 & \tf 1.0064654716056596 & \tf 0.4467177201362911 & \tf 1.0064654716056596 & \tf 0.4467177201362911 \\\hline
 4 & \tf 1.5002730728687338 & \tf 0.5330387209380418 & \tf 1.5002730728687338 & \tf 0.5330387209380418 \\\hline
 5 & \tf 2.0000387972956304 & \tf 0.4665398713719121 & \tf 2.0000387972956304 & \tf 0.4665398713719121 \\\hline
 6 & \tf 2.5000000105321137 & \tf 0.5333333220982075 & \tf 2.5000000105321137 & \tf 0.5333333220982075 \\\hline
 7 & \tf 3.3.00000000150452 & \tf 0.4666666617518435 & \tf 3.0000000015045293 & \tf 0.4666666617518435 \\\hline
 8 & \tf 3.5                & \tf 0.5333333333333333 & \tf 3.5000000000000000 & \tf 0.5333333333333333 \\\hline
 9 & --                     & --                     & \tf 4                  & \tf 0.4666666666666665\\\hline\hline
    & \multicolumn{2}{c|}{$n=9$} & \multicolumn{2}{c|}{$n=10$}\\
 1 & \tf 0.1225148226554413 & \tf 0.3020174288145723 & \tf 0.1225148226554413 & \tf 0.3020174288145723 \\\hline
 2 & \tf 0.5441518440112252 & \tf 0.4850196082224646 & \tf 0.5441518440112252 & \tf 0.4850196082224646 \\\hline
 3 & \tf 1.0064654716056596 & \tf 0.4467177201362911 & \tf 1.0064654716056596 & \tf 0.4467177201362911 \\\hline
 4 & \tf 1.5002730728687338 & \tf 0.5330387209380418 & \tf 1.5002730728687338 & \tf 0.5330387209380418 \\\hline
 5 & \tf 2.0000387972956304 & \tf 0.4665398713719121 & \tf 2.0000387972956304 & \tf 0.4665398713719121 \\\hline
 6 & \tf 2.5000000105321137 & \tf 0.5333333220982075 & \tf 2.5000000105321137 & \tf 0.5333333220982075 \\\hline
 7 & \tf 3.0000000015045293 & \tf 0.4666666617518435 & \tf 3.0000000015045293 & \tf 0.4666666617518435 \\\hline
 8 & \tf 3.5000000000000000 & \tf 0.5333333333333333 & \tf 3.5000000000000000 & \tf 0.5333333333333333 \\\hline
 9 & \tf 4.0000000000000000 & \tf 0.4666666666666666 & \tf 4.0000000000000000 & \tf 0.4666666666666665 \\\hline
10 & \tf 4.5                & \tf 0.5333333333333333 & \tf 4.5000000000000000 & \tf 0.5333333333333333 \\\hline
11 & --                     & --                     & \tf 5.0000000000000000 & \tf 0.4666666666666666 \\\hline
\end{tabular}
}
\end{center}
\end{table}

\section{Conclusion and future work}\label{sec:con} %

We have presented a recursive algorithm that computes quadrature nodes and weights
for spaces of quintic splines with uniform knot sequences over finite domains.
The presented quadrature is explicit, that is,
in every step of the recursion the new nodes and weights are computed in closed form,
without using a numerical solver. The number of nodes per subinterval is two
and hence the cost reduction compared to the classical Gaussian quadrature for polynomials is $2/3$.
We have also shown numerically that in the limit, when the length of the interval goes to infinity,
our rule converges to the ``two-third rule'' of Hughes et al. \cite{Hughes-2010}
that is known to be exact and optimal over the real line.

For $C^1$ splines, our quadrature rule is optimal (Gaussian), that is, it requires minimal number of evaluations.
However, it is still exact for any quintic splines
with higher continuity and therefore we believe it can be used for various engineering applications.

As a future work, optimal rules for spaces of higher polynomial degree and primarily of higher continuity
are within the scope of our interest. Since the higher continuity reduces the number of optimal nodes,
their layout in subintervals is difficult to assign and hence finding these is still an open problem.


\bibliographystyle{plain}
\bibliography{QuinticC1QuadratureArXiv}

\section*{Appendix}


\begin{proof}[Proof of Lemma~\ref{lem:NoRoot}]
Without loss of generality, we may assume $h=1$ as the roots of $C_k$ change with scaling factor $h$, cf. (\ref{eq:cubic}).
The cubic $C_k$ can be split into two summands $f_k$ and $g_k$, the first independent and the latter dependent on $r_{4k-3}$
and $r_{4k-2}$, 
i.e.
\begin{equation}\label{eq:summands}
\begin{split}
f_k & = 2t^3 -5t^2 +4t - 1 = (t-1)^2 (t-\frac{1}{2}),\\
g_k & = (-216r_{4k-3} - 24r_{4k-2})t^3 + (24r_{4k-2} -24r_{4k-3})t^2 \\
    & = -24t^2(t- \frac{r_{4k-2} - r_{4k-3}}{9r_{4k-3} + r_{4k-2}}) 
\end{split}
\end{equation}
We show that $C_k(t) < 0$ on $[0,1]$. We denote by $\xi_g$ the non-zero root of $g_k$,
$$\xi_g = \frac{r_{4k-2} - r_{4k-3}}{9r_{4k-3} + r_{4k-2}},$$
and consider the particular subdivision of $[0,1]$ into 1) $[0,\xi_g]$,
2) $[\xi_g, \frac{1}{2}]$, and 3) $[\frac{1}{2},1]$.
We investigate $C_k$ on each of these three subintervals separately:

Case 3) We express $C_k$ in Bernstein (B) basis and show all its coefficients are negative.
Let $T_{[\frac{1}{2},1]}$ be the transformation matrix \cite{Hoschek-2002-CAGD} from monomial to Bernstein basis on $[\frac{1}{2},1]$
\begin{equation}\label{eq:Matrix}
\renewcommand{\arraystretch}{1.2}
 T_{[\frac{1}{2},1]}
  =  \left(\begin{matrix}
      1 & 1 & 1 & 1\\
      \frac{1}{2}& \frac{2}{3}& \frac{5}{6} & 1  \\
      \frac{1}{4}& \frac{5}{12}& \frac{2}{3} & 1  \\
      \frac{1}{8}& \frac{1}{4}& \frac{1}{2} & 1  \\
   \end{matrix}\right),
\end{equation}
and let $\bc_B^k = (c_0^B, c_0^B, c_0^B, c_0^B)$ be the vector of Bernstein coefficients.
Then the conversion is given by $\bc^{k}_B = \bc^{k}_{mo}  T_{[\frac{1}{2},1]}$
and we obtain
%
\begin{equation}\label{eq:BBcoords}
\begin{split}
& c^B_0 = -33r_{4k-3}+3r_{4k-2},\\
& c^B_1 = \frac{1}{12} -64r_{4k-3}+4r_{4k-2},\\
& c^B_2 =  -124r_{4k-3}+4r_{4k-2},\\
& c^B_3 =  -240r_{4k-3}.
\end{split}
\end{equation}
Looking at the $c_i^B$ coefficients, we start with the second one and prove that
$64r_{4k-3} - 4r_{4k-2} > \frac{1}{12}$. We know that $r_{4k-3}$
and $r_{4k-2}$ are by definition both positive and also $r_{4k-2}>r_{4k-3}$,
which is a direct consequence of $D_{4k+1}(t) > D_{4k+2}(t)$ on $(x_{k-1},x_k)$,
and also $r_{4k-3},r_{4k-2}<\frac{1}{6}$.
Consider the blend $P_k(t)$ from Lemma~\ref{lemmaBlend}. $P_k(t)\geq0$
gives polynomial inequality $2 D_{4k+1}(t) - 2 D_{4k+2}(t) + \frac{1}{2} D_{4k-1}(t) \geq D_{4k}(t)$.
By evaluating both sides at the two nodes and by multiplying with corresponding weights, i.e., by applying the
quadrature rule $\Q$ on $[x_{k-1},x_k]$, we obtain
\begin{equation}\label{eq:ruleonQ}
2(4r_{4k-3} - r_{4k-2}) + \frac{1}{12} \geq \frac{1}{6}
\end{equation}
because $D_{4k-1}$ and $D_{4k}$ act only on this interval and hence the rule reproduces their
integrals exactly. Combining (\ref{eq:ruleonQ}) with $r_{4k-3}>0$ proves the desired inequality.
Moreover, combining (\ref{eq:ruleonQ}) with $\frac{1}{6}>r_{4k-2}$ gives
\begin{equation}\label{eq:r1r2}
16r_{4k-3} > 5r_{4k-2},
\end{equation}
and the other three inequalities follow directly from (\ref{eq:r1r2}) and the fact $r_{4k-3}>0$.

Case 1) Similarly to case 3), we compute the $c_i^B$ coefficients and show that all are negative.
The conversion is given by $\bc^k_B = \bc^k_{mo}  T_{[0,\xi_g]}$, where
\begin{equation}\label{eq:Matrix2}
\renewcommand{\arraystretch}{1.2}
 T_{[0,\xi_g]}
  =  \left(\begin{matrix}
      1 & 1 & 1 & 1\\
      0 & \frac{1}{3}\xi_g &  \frac{1}{2}\xi_g & \xi_g  \\
      0 & 0   &  \frac{1}{3}\xi_g^2 & \xi_g^2  \\
      0 & 0   & 0     & \xi_g^3  \\
   \end{matrix}\right)
\end{equation}
is the corresponding transformation matrix. We obtain $\bc^{k}_B = (c_0^B,c_1^B,c_2^B,c_3^B)$
\begin{equation}\label{eq:cubic2}
\begin{split}
& c^B_0 = -1,\\
& c^B_1 = \frac{-31r_{4k-3} + r_{4k-2}}{3(9r_{4k-3} + r_{4k-2})},\\
& c^B_2 = -\frac{4}{3}\frac{80r_{4k-3}^2 - 5r_{4k-3}r_{4k-2} + 6(r_{4k-3}-r_{4k-2})^3}{(9r_{4k-3} + r_{4k-2})^2},\\
& c^B_3 = \frac{-100r_{4k-3}^2 (11r_{4k-3} - r_{4k-2})}{(9r_{4k-3} + r_{4k-2})^3}.
\end{split}
\end{equation}
Negativity of $c^B_1$ and $c^B_3$ follows directly from (\ref{eq:r1r2}). It remains to prove
$$
80r_{4k-3}^2 - 5r_{4k-3}r_{4k-2} + 6(r_{4k-3}-r_{4k-2})^3 > 0,
$$
which using (\ref{eq:r1r2}) and $r_{4k-2}>r_{4k-3}$ simplifies to
$$
64r_{4k-3}^2 > 6r_{4k-2}^3,
$$
which again follows from (\ref{eq:r1r2}) using $r_{4k-2}<\frac{1}{6}$.

Case 2) $f_k$ has a double root at $t=1$ and $g_k$ has a double root at $t=0$.
For the third root, it holds
\begin{equation}\label{eq:root}
\frac{r_{4k-2} - r_{4k-3}}{9r_{4k-3} + r_{4k-2}} < \frac{1}{2}
\end{equation}
which follows from (\ref{eq:r1r2}). From
case 3), we know $g_k(\frac{1}{2}) = -33r_{4k-3} + 3r_{4k-2}$ and from case 1) $f_k(\xi_g)=c_3^B$
that are both negative.
Moreover, we know that $f_k$ is monotonically increasing while $g_k$ is monotonically
decreasing on $(\xi_g, \frac{1}{2})$. 
Therefore $f_k(t)<0$ on $[\xi_g, \frac{1}{2})$
and $g_k<0$ on $(\xi_g, \frac{1}{2}]$ which completes the proof.
\end{proof}


\end{document}